\documentclass[12pt]{amsart}

\usepackage{amsfonts}

\usepackage{amssymb}
\usepackage{latexsym}
\usepackage{graphics}
\usepackage[2emode]{psfrag}
\usepackage{mathrsfs}
\usepackage{amsthm}
\usepackage{amsmath}
\usepackage[all]{xy}
\usepackage{enumerate}
\usepackage{hyperref}
\usepackage{stmaryrd}
\usepackage{fancyhdr}
\usepackage[table]{xcolor}
\usepackage{cite}
\usepackage{color}
\usepackage{tikz}
\usepackage{tqft}
\usetikzlibrary{tqft}
\usetikzlibrary{calc,tqft,decorations.markings}
\usetikzlibrary{arrows, decorations.markings}
\usetikzlibrary{calc, arrows, decorations.markings}
\newtheorem{mainthm}{Theorem}

\usepackage{freetikz}

\begin{document}

\def\joost{\color{blue}}

\newcommand{\thlabel}[1]{\label{th:#1}}
\newcommand{\thref}[1]{Theorem~\ref{th:#1}}
\newcommand{\selabel}[1]{\label{se:#1}}
\newcommand{\seref}[1]{Section~\ref{se:#1}}
\newcommand{\lelabel}[1]{\label{le:#1}}
\newcommand{\leref}[1]{Lemma~\ref{le:#1}}
\newcommand{\prlabel}[1]{\label{pr:#1}}
\newcommand{\prref}[1]{Proposition~\ref{pr:#1}}
\newcommand{\colabel}[1]{\label{co:#1}}
\newcommand{\coref}[1]{Corollary~\ref{co:#1}}
\newcommand{\relabel}[1]{\label{re:#1}}
\newcommand{\reref}[1]{Remark~\ref{re:#1}}
\newcommand{\exlabel}[1]{\label{ex:#1}}
\newcommand{\exref}[1]{Example~\ref{ex:#1}}
\newcommand{\delabel}[1]{\label{de:#1}}
\newcommand{\deref}[1]{Definition~\ref{de:#1}}
\newcommand{\eqlabel}[1]{\label{eq:#1}}
\newcommand{\equref}[1]{(\ref{eq:#1})}

\newcommand{\Hom}{{\sf Hom}}
\newcommand{\End}{{\sf End}}
\newcommand{\Ext}{{\sf Ext}}
\newcommand{\Fun}{{\sf Fun}}
\newcommand{\Mor}{{\sf Mor}\,}
\newcommand{\Aut}{{\sf Aut}\,}
\newcommand{\Ann}{{\sf Ann}\,}
\newcommand{\Ker}{{\sf Ker}\,}
\newcommand{\Coker}{{\sf Coker}\,}
\newcommand{\Cat}{{\sf Cat}\,}
\newcommand{\opCat}{{\sf opCat}\,}
\newcommand{\im}{{\sf Im}\,}
\newcommand{\coim}{{\sf Coim}\,}
\newcommand{\Trace}{{\sf Trace}\,}
\newcommand{\Frob}{{\sf Frob}\,}
\newcommand{\Char}{{\sf Char}\,}
\newcommand{\Mod}{{\sf Mod}}
\newcommand{\Vect}{{\sf Vect}}
\newcommand{\VectGr}{{\sf Vect-Grph}}
\newcommand{\Alg}{{\sf Alg}}
\newcommand{\Coalg}{{\sf Coalg}}
\newcommand{\Bialg}{{\sf Bialg}}
\newcommand{\Hopf}{{\sf Hopf}}
\newcommand{\Spec}{{\sf Spec}\,}
\newcommand{\Span}{{\sf Span}\,}
\newcommand{\sgn}{{\sf sgn}\,}
\newcommand{\Id}{{\sf Id}\,}
\newcommand{\Com}{{\sf Com}\,}
\newcommand{\codim}{{\sf codim}}
\newcommand{\Mat}{{\sf Mat}}
\newcommand{\Coint}{{\rm Coint}}
\newcommand{\Incoint}{{\sf Incoint}}
\newcommand{\can}{{\sf can}}
\newcommand{\Bim}{{\sf Bim}}
\newcommand{\CAT}{{\sf CAT}}
\newcommand{\Ob}{{\sf Ob}}
\newcommand{\sign}{{\sf sign}}
\newcommand{\kar}{{\sf kar}}
\newcommand{\rad}{{\sf rad}}
\newcommand{\Rat}{{\sf Rat}}
\newcommand{\Cob}{{\sf Cob}}
\newcommand{\hCob}{{\sf hCob}}
\newcommand{\ev}{{\sf ev}}
\newcommand{\coev}{{\sf coev}}
\newcommand{\sd}{{\sf d}}
\def\colim{{\sf colim}\,}
\def\tildej{\tilde{\jmath}}
\def\barj{\bar{\jmath}}

\def\Ab{\underline{\underline{\sf Ab}}}
\def\lan{\langle}
\def\ran{\rangle}
\def\ot{\otimes}
\def\bul{\bullet}
\def\ubul{\underline{\bullet}}

\def\id{\textrm{{\small 1}\normalsize\!\!1}}
\def\To{{\multimap\!\to}}
\def\bigperp{{\LARGE\textrm{$\perp$}}} 
\newcommand{\QED}{\hspace{\stretch{1}}
\makebox[0mm][r]{$\Box$}\\}

\def\AA{{\mathbb A}}
\def\BB{{\mathbb B}}
\def\CC{{\mathbb C}}
\def\DD{{\mathbb D}}
\def\EE{{\mathbb E}}
\def\FF{{\mathbb F}}
\def\GG{{\mathbb G}}
\def\HH{{\mathbb H}}
\def\II{{\mathbb I}}
\def\JJ{{\mathbb J}}
\def\KK{{\mathbb K}}
\def\LL{{\mathbb L}}
\def\MM{{\mathbb M}}
\def\NN{{\mathbb N}}
\def\OO{{\mathbb O}}
\def\PP{{\mathbb P}}
\def\QQ{{\mathbb Q}}
\def\RR{{\mathbb R}}
\def\TT{{\mathbb T}}
\def\UU{{\mathbb U}}
\def\VV{{\mathbb V}}
\def\WW{{\mathbb W}}
\def\XX{{\mathbb X}}
\def\YY{{\mathbb Y}}
\def\ZZ{{\mathbb Z}}

\def\aa{{\mathfrak A}}
\def\bb{{\mathfrak B}}
\def\cc{{\mathfrak C}}
\def\dd{{\mathfrak D}}
\def\ee{{\mathfrak E}}
\def\ff{{\mathfrak F}}
\def\gg{{\mathfrak G}}
\def\hh{{\mathfrak H}}
\def\ii{{\mathfrak I}}
\def\jj{{\mathfrak J}}
\def\kk{{\mathfrak K}}
\def\ll{{\mathfrak L}}
\def\mm{{\mathfrak M}}
\def\nn{{\mathfrak N}}
\def\oo{{\mathfrak O}}
\def\pp{{\mathfrak P}}
\def\qq{{\mathfrak Q}}
\def\rr{{\mathfrak R}}
\def\tt{{\mathfrak T}}
\def\uu{{\mathfrak U}}
\def\vv{{\mathfrak V}}
\def\ww{{\mathfrak W}}
\def\xx{{\mathfrak X}}
\def\yy{{\mathfrak Y}}
\def\zz{{\mathfrak Z}}

\def\aaa{{\mathfrak a}}
\def\bbb{{\mathfrak b}}
\def\ccc{{\mathfrak c}}
\def\ddd{{\mathfrak d}}
\def\eee{{\mathfrak e}}
\def\fff{{\mathfrak f}}
\def\ggg{{\mathfrak g}}
\def\hhh{{\mathfrak h}}
\def\iii{{\mathfrak i}}
\def\jjj{{\mathfrak j}}
\def\kkk{{\mathfrak k}}
\def\lll{{\mathfrak l}}
\def\mmm{{\mathfrak m}}
\def\nnn{{\mathfrak n}}
\def\ooo{{\mathfrak o}}
\def\ppp{{\mathfrak p}}
\def\qqq{{\mathfrak q}}
\def\rrr{{\mathfrak r}}
\def\sss{{\mathfrak s}}
\def\ttt{{\mathfrak t}}
\def\uuu{{\mathfrak u}}
\def\vvv{{\mathfrak v}}
\def\www{{\mathfrak w}}
\def\xxx{{\mathfrak x}}
\def\yyy{{\mathfrak y}}
\def\zzz{{\mathfrak z}}

\newcommand{\aA}{\mathscr{A}}
\newcommand{\bB}{\mathscr{B}}
\newcommand{\cC}{\mathscr{C}}
\newcommand{\dD}{\mathscr{D}}
\newcommand{\eE}{\mathscr{E}}
\newcommand{\fF}{\mathscr{F}}
\newcommand{\gG}{\mathscr{G}}
\newcommand{\hH}{\mathscr{H}}
\newcommand{\iI}{\mathscr{I}}
\newcommand{\jJ}{\mathscr{J}}
\newcommand{\kK}{\mathscr{K}}
\newcommand{\lL}{\mathscr{L}}
\newcommand{\mM}{\mathscr{M}}
\newcommand{\nN}{\mathscr{N}}
\newcommand{\oO}{\mathscr{O}}
\newcommand{\pP}{\mathscr{P}}
\newcommand{\qQ}{\mathscr{Q}}
\newcommand{\rR}{\mathscr{R}}
\newcommand{\sS}{\mathscr{S}}
\newcommand{\tT}{\mathscr{T}}
\newcommand{\uU}{\mathscr{U}}
\newcommand{\vV}{\mathscr{V}}
\newcommand{\wW}{\mathscr{W}}
\newcommand{\xX}{\mathscr{X}}
\newcommand{\yY}{\mathscr{Y}}
\newcommand{\zZ}{\mathscr{Z}}

\newcommand{\Aa}{\mathcal{A}}
\newcommand{\Bb}{\mathcal{B}}
\newcommand{\Cc}{\mathcal{C}}
\newcommand{\Dd}{\mathcal{D}}
\newcommand{\Ee}{\mathcal{E}}
\newcommand{\Ff}{\mathcal{F}}
\newcommand{\Gg}{\mathcal{G}}
\newcommand{\Hh}{\mathcal{H}}
\newcommand{\Ii}{\mathcal{I}}
\newcommand{\Jj}{\mathcal{J}}
\newcommand{\Kk}{\mathcal{K}}
\newcommand{\Ll}{\mathcal{L}}
\newcommand{\Mm}{\mathcal{M}}
\newcommand{\Nn}{\mathcal{N}}
\newcommand{\Oo}{\mathcal{O}}
\newcommand{\Pp}{\mathcal{P}}
\newcommand{\Qq}{\mathcal{Q}}
\newcommand{\Rr}{\mathcal{R}}
\newcommand{\Ss}{\mathcal{S}}
\newcommand{\Tt}{\mathcal{T}}
\newcommand{\Uu}{\mathcal{U}}
\newcommand{\Vv}{\mathcal{V}}
\newcommand{\Ww}{\mathcal{W}}
\newcommand{\Xx}{\mathcal{X}}
\newcommand{\Yy}{\mathcal{Y}}
\newcommand{\Zz}{\mathcal{Z}}

\def\units{{\mathbb G}_m}
\def\rightact{\hbox{$\leftharpoonup$}}
\def\leftact{\hbox{$\rightharpoonup$}}

\def\*C{{}^*\hspace*{-1pt}{\Cc}}
\def\*c{{}^*\hspace*{-1pt}{\Cc}}

\def\text#1{{\rm {\rm #1}}}

\def\smashco{\mathrel>\joinrel\mathrel\triangleleft}
\def\cosmash{\mathrel\triangleright\joinrel\mathrel<}

\def\ol{\overline}
\def\ul{\underline}
\def\dul#1{\underline{\underline{#1}}}
\def\Nat{\dul{\rm Nat}}
\def\Set{\dul{\rm Set}}
\def\MCl{\dul{\rm MCl}}

\renewcommand{\subjclassname}{\textup{2000} Mathematics Subject
     Classification}

\newtheorem{proposition}{Proposition}[section] 
\newtheorem{lemma}[proposition]{Lemma}
\newtheorem{corollary}[proposition]{Corollary}
\newtheorem{theorem}[proposition]{Theorem}

\theoremstyle{definition}
\newtheorem{Definition}[proposition]{Definition}
\newtheorem{example}[proposition]{Example}
\newtheorem{examples}[proposition]{Examples}

\theoremstyle{remark}
\newtheorem{remarks}[proposition]{Remarks}
\newtheorem{remark}[proposition]{Remark}

\newenvironment{proofsketch}{%
\renewcommand{\proofname}{Proof sketch}\proof}{\endproof}

\title{2D HQFTs and Frobenius $(\Gg,\Vv)$-categories}

\author[P. Gro{\ss}kopf]{Paul Gro{\ss}kopf}
\address{Paul Gro{\ss}kopf, Mathematical Institute, University of Oxford, United Kingdom}
\email{paul.grosskopf@gmx.at}

\maketitle
\begin{abstract}
Homotopy Quantum Field Theories as variants of Topological Quantum Field Theories are described by functors from some cobordism category, enriched with homotopical data, to a symmetric monoidal category $\Vv$. A new notion of HQFTs is introduced using target pairs of spaces $(X,Y)$ acounting for basepoints being sent to points in $Y$. Such $(X,Y)$-HQFTs are classified in dimension 1 by dualizable representations of $\Gg:=\Pi_1(X,Y)$, the relative fundamental groupoid. For dimension 2, the notion of crossed loop Frobenius $(\Gg,\Vv)$-categories is introduced, generalizing crossed Frobenius $G$-algebras, where $G$ is only a group. After stating generalities of these multi-object generalizations, a classification theorem of 2-dimensional $(X,Y)$-HQFTs via crossed loop Frobenius $(\Gg,\Vv)$-categories is proven. 
\end{abstract}

\tableofcontents

\section{Introduction}
Homotopy Quantum Field Theories were introduced by Turaev in \cite{Tu} and generalize the notion of Topological Quantum Field Theories. These can be seen as invariants of manifolds, that behave functorially with respect to gluings of manifolds along the boundary. One therefore considers the category $\Cob_{d+1}$, which has $d$-dimensional closed oriented manifolds as objects and compact oriented $(d+1)$-dimensional cobordisms as morphisms between them. This category further is endowed with a symmetric monoidal structure, given by the disjoint union. A TQFT can now be described as a symmetric monoidal functor from $\Cob_{d+1}$ to the category $\Vect_\KK$ of vector spaces over a field $\KK$ or more generally any symmetric monoidal \textit{target category} $\Vv$. In dimension 2, TQFTs are fully classified by commutative Frobenius monoids in $\Vv$ \cite{Ab}. \\
For HQFTs we fix a pointed space $X$ and endow all manifolds with maps to this \textit{target space} considered up to homotopy. One again can define a category $\Cob^X_{d+1}$ and define $X$-HQFTs to be symmetric monoidal functors $\Cob^X_{d+1}\to \Vect_\KK$. In dimension 1, such $X$-HQFTs with target space $X$ are classified by finite dimensional representations of $\pi_1(X)$. Futhermore, for contractible spaces, the theory of $X$-HQFTs specifies to the theory of TQFTs. If $X$ is homotopically more complex we get more involved algebraic structures. Particularly, if $X=K(G,1)$ is the Eilenberg-MacLane space of a group $G$, then 2D $X$-HQFTs are fully classified by so-called crossed Frobenius $G$-algebras \cite{Tu}. For homotopy 2-types work has been conducted in the cases $X=K(A,2)$, for $A$ an abelian group, \cite{BT} as well as the general case where both $\pi_1(X)$ and $\pi_2(X)$ are non trivial \cite{ST, PT}. Furthermore, it is known that $(d+1)$-dimensional $X$-HQFTs only depend on the homotopy $(d+1)$-type of the target space \cite{R}. \\
However, in this article I want to explore a more general notion of HQFTs with target $X$ a homotopy 1-type. Classically we consider cobordisms whose bounday components are pointed and pointed maps to $X$. We can loosen these restrictions and consider some subset $Y\subset X$, whose elements are possible images of the base points. This leads to the notion of $(X,Y)$-HQFTs as symmetric monoidal functors 
$$\Cob^{X,Y}_{d+1}\to \Vv.$$
We can classify these easily in dimension 1 by dualizable representations of the relative fundamental groupoid $\Gg=\Pi_1(X,Y)$ in $\Vv$, assuming some connectedness condition on the pair $(X,Y)$. 

\begin{mainthm}
Let $\Vv$ be a (strict) symmetric monoidal category. Further let $X$ be a homotopy 1-type, $Y\subset X$, such that for any connected component $C$ of $X$ we have $C\cap Y\neq \emptyset$, and denote $\Gg= \Pi_1(X,Y)$. The 1-dimensional $(X,Y)$-HQFTs 
$$\Zz: \Cob^{X,Y}_1 \to \Vv$$
 are fully classified by functors 
$$\Gg \to \Vv,$$
such that all objects in the image are dualizable.
\end{mainthm}

Moreover, the main result of this article is a classification of 2D $(X,Y)$-HQFTs via so-called \textit{crossed loop Frobenius $(\Gg,\Vv)$-categories}. I will give two equivalent definitions of Frobenius $(\Gg,\Vv)$-categories in analogue to the characterisation of Frobenius algebras via comultiplication and counit or via a non-degenerate pairing. I further introduce crossings and proof the main classification theorem of this article. 

\begin{mainthm}
Let $\Vv$ be a (strict) symmetric monoidal category. Further, let $X$ be a homotopy 1-type, $Y\subset X$, such that for any connected component $C$ of $X$ we have $C\cap Y\neq \emptyset$, and denote $\Gg= \Pi_1(X,Y)$. Then the 2-dimensional $(X,Y)$-HQFTs
$$\Zz: \Cob^{X,Y}_2\to \Vv$$
are fully classified by crossed loop Frobenius $(\Gg,\Vv)$-categories.
\end{mainthm}

As indicated by the nomenclature this theorem can be thought of a groupoid-version of Turaev's original result. Particularly, choosing $\Vv=\Vect_\KK$ and $Y=\{*\}$ one retrieves the classification of $X$-HQFTs via crossed Frobenius $G$-algebras. 

\subsection*{Acknowledgements}
I want to thank my PhD supervisor Joost Vercruysse at Université Libre de Bruxelles, who gave me the intial idea to work on this project. Further, I want to thank Alexis Virelizier, Nils Carqueville and André Henriques for valuable discussions and challenging me to be more precise and careful at times. 
I was supported by a FRIA fellowship of FNRS FNRS (Fonds National de la Recherche Scientifique), grant number FC41285, while working on this project. Furthermore, while finishing this article I am supported by the FWA (Foundation Wiener-Anspach), recieving a postdoctoral fellowship.

\section{Prelimaries and Conventions}
In this section we discuss elementary notions and conventions used in the rest of the article. 
\subsection{Monoidal categories}
\begin{Definition}
A \textit{\textbf{monoidal category}} $(\Vv,\ot,I,\alpha,\lambda,\rho)$  consists of a category $\Vv$, an object $I$ of $\Vv$ called \textit{\textbf{monoidal unit}} and bifunctor 
$$\ot: \Vv \times \Vv \to \Vv$$
 called \textit{\textbf{tensor product}}. Further, there are natural isomorphisms
$$\alpha: \ot(\Vv\times \ot) \to \ot(\ot \times \Vv),$$
 called the \textit{\textbf{associator}}, and 
$$\lambda: \ot (I\times \Vv) \to \Vv, \qquad \rho: \ot(\Vv\times I)\to \Vv,$$
called the \textit{\textbf{left and right unitor}}, satisfying the \textit{Pentagram rule}  for all objects $A,B,C,D$ of $\Vv$:
\[
\xymatrix{
 & (A\ot B)\ot(C \ot D)  \ar[rd]^{\alpha_{A \ot B,C,D}} & \\
A\ot(B\ot(C \ot D)) \ar[ru]^{\alpha_{A,B,C\ot D}}  \ar[d]_{A \ot \alpha_{B,C,D}}& &((A\ot B)\ot C) \ot D  \ar[d]^{\alpha_{A,B,C}\ot D} \\
A\ot((B\ot C) \ot D) \ar[rr]_{\alpha_{A,B\ot C,D}} & & (A\ot(B\ot C)) \ot D\\
}
\]
and 
\[
\xymatrix{
A\ot(I\ot C)  \ar[r]^{\alpha_{A,I,C}}  \ar[d]_{A \ot \lambda_C} & (A\ot I)\ot C \ar[d]^{ \rho_A \ot C}\\
A\ot C \ar@{=}[r] & A\ot C.
}
\]
If $\alpha,\lambda,\rho$ are identities, the monoidal category is called \textit{\textbf{strict}}.
\end{Definition}

\begin{Definition}
Let $(\Vv,\ot, I, \alpha, \lambda, \rho)$ and $(\Vv',\boxtimes, I', \alpha', \lambda', \rho')$ be two monoidal categories. A \textit{\textbf{monoidal functor}} between $\Vv$ and $\Vv'$ is a triple $(F,\phi, \phi_0),$ where $F:\Vv\to \Vv'$ is a functor, $\phi: F\boxtimes F \to F (\Vv\ot \Vv)$ a natural transformation and $\phi_0: I'\to FI$ a morphism in $\Vv'$ making the following diagrams commute:
\[
\xymatrix{
FA \boxtimes (FB \boxtimes FC) \ar[rr]^{\alpha'_{FA,FB,FC}} \ar[d]_{FA\boxtimes \phi_{B,C}} & & (FA \boxtimes FB) \boxtimes FC \ar[d]^{\phi_{A,B}\boxtimes FC} \\
FA \boxtimes F(B \ot C)\ar[d]_{\phi_{A,B\ot C}} & &F(A \ot B) \boxtimes FC \ar[d]^{\phi_{A\ot B,C}} \\
F(A \ot (B \ot C)) \ar[rr]_{F\alpha_{A,B,C}} & & F((A \ot B) \ot C), \\
}
\]
\[
\xymatrix{
FB \boxtimes I' \ar[r]^-{\rho'_{FB}} \ar[d]_{FB \boxtimes \phi_0} & FB & I'\boxtimes FB \ar[r]^-{\lambda'_{FB}} \ar[d]_{\phi_0 \boxtimes FB} &  FB\\
FB \boxtimes FI \ar[r]_{\phi_{B,I}} & F( B \ot I), \ar[u]_{F\rho_B} & FI \boxtimes FB \ar[r]_{\phi_{I,B}} & F( I \ot B). \ar[u]_{F\lambda_B} \\
}
\]
 
If $\phi$ and $\phi_0$ are identities, the monoidal functor is called \textit{\textbf{strict}}. 
\end{Definition}

\begin{Definition}
Let $(\Vv,\ot, I)$ and $(\Vv',\boxtimes, I')$ be two monoidal categories and $(F,\phi, \phi_0)$ and $(G,\psi, \psi_0)$ two monoidal functors $\Vv\to \Vv'$. A natural transformation $\rho: F\Rightarrow G$ is called \textit{\textbf{monoidal}}, if 

\[
\xymatrix{
FA \boxtimes FB \ar[r]^{\rho_A \boxtimes \rho_B} \ar[d]_{\phi_{A,B}} & GA \boxtimes GB \ar[d]^{\psi_{A,B}} & I' \ar[r]^{\psi_0} \ar[d]_{\phi_0} & GI \\
F(A \ot B)  \ar[r]_{\rho_{A\ot B}} & G(A \ot B), &FI \ar[ru]_{\rho_I} & \\}
\]
are commutative diagrams. 
\end{Definition}

We define \textit{braided} and \textit{symmetric} monoidal categories. These give us the right setting to talk for example about any commutative and cocommutative phenomena.

\begin{Definition}
A \textit{\textbf{braiding}} of a monoidal category $\Vv$ is a natural isomorphims $\sigma_{A,B}: A\ot B \to B \ot A$ such that $\lambda_A \sigma_{I,A}= \rho _A$ for all objects $A$ of $\Vv$ and the diagrams
\[
\xymatrix{
(A\ot B) \ot C \ar[r]^{\sigma_{A\ot B,C}} & C\ot (A\ot B) \ar[d]^{\alpha_{C,A,B}} & A\ot( B \ot C) \ar[r]^{\sigma_{A, B \ot C}} & (B\ot C) \ot A \ar[d]^{\alpha_{B,C,A}^{-1}}\\
 A \ot (B\ot C) \ar[u]^{\alpha_{A,B,C}} \ar[d]_{A\ot \sigma_{B,C}} & (C\ot A)\ot B  & (A \ot B)\ot C, \ar[u]^{\alpha^{-1}_{A,B,C}} \ar[d]_{\sigma_{A,B}\ot C} & B\ot (C\ot A)  \\
A \ot (C \ot B) \ar[r]_{\alpha_{A,C,B}} & (A\ot C)\ot B, \ar[u]_{\sigma_{A,C}\ot B} & (B \ot A) \ot C \ar[r]_{\alpha^{-1}_{B,A,C}} & B\ot (A\ot C) \ar[u]_{B \ot \sigma_{A,C}} ,\\
}
\]
commute for all objects $A,B,C$ in $\Vv$.\\
A monoidal category is called \textit{\textbf{braided}}, if it has a braiding $\sigma$ and \textit{\textbf{symmetric}}, if $\sigma^{-1}_{A,B}=\sigma_{B,A}$ for all objects $A,B$ in $\Vv$.\\
A monoidal functor $(F,\phi,\phi_0)$ between two braided monoidal categories  $(\Vv,\ot, I, \alpha, \lambda, \rho, \sigma)$ and $(\Vv',\boxtimes, I', \alpha', \lambda', \rho',\sigma')$ is called \textit{\textbf{braided}} if 
\[
\xymatrix{
FA\boxtimes FB \ar[r]^{\sigma'} \ar[d]_{\phi_{A,B}} & FB \boxtimes FA \ar[d]^{\phi_{B,A}}\\
F(A\ot B) \ar[r]_{F\sigma} & F(B\ot A)}
\]
are commutative diagrams for all $A,B$ in $\Vv$. A \textit{\textbf{symmetric}} monoidal functor is a braided monoidal functor between symmetric monoidal categories.
\end{Definition}

\begin{Definition}
Let $(\Vv,\ot, I)$ and $(\Vv',\boxtimes, I')$ be two symmetric monoidal categories. They are \textit{\textbf{symmetric monoidally equivalent}}, if there exist symmetric monoidal functors $(F,\phi, \phi_0): \Vv\to \Vv'$ and $(G,\psi, \psi_0):\Vv'\to\Vv$ and monoidal natural isomorphisms $\epsilon: FG \Rightarrow \Id_{\Vv'}$ and $\eta:\Id_{\Vv}\Rightarrow GF$. 
\end{Definition}

\begin{Definition} 
Let $(\Vv,\ot, I)$ and $(\Vv',\boxtimes, I')$ be two monoidal categories. The \textit{\textbf{symmetric monoidal functor category}} is defined as the full subcategory $\Fun^{\ot,s}(\Vv,\Vv')\subseteq\Fun^{\ot}(\Vv,\Vv')$ with objects symmetric monoidal functors and morphisms monoidal natural transformations. 
\end{Definition}

The following proposition is well know and is therefore only stated without proof.
\begin{proposition} \label{funcateq2}
Let $\Cc$ be a symmetric monoidal category and $\Vv$, $\Vv'$ two symmetric monoidally equivalent categories. Then the monoidal functor categories $\Fun^{\ot,s}(\Vv,\Cc)$ and $\Fun^{\ot,s}(\Vv',\Cc)$ are equivalent as categories. 
\end{proposition}


In order to avoid writing down associators and unitors, we might consider MacLane's coherence theorem for symmetric monoidal categories:
\begin{theorem}[{\cite[Thm.1. in XI.3.]{MacLane} as well as \cite[Thm.XI.5.3.]{Kassel}}]
Any symmetric monoidal category $\Vv$ is symmetric monoidally equivalent to a strict symmetric monoidal category $\Ss$.
\end{theorem}
Henceforth, we assume every symmetric monoidal category to be strict. 

\subsection{Dualizability and Rigid categories}
\begin{Definition}
Let $(\Vv,\ot,I)$ be a monoidal category and $X$ an object of $\Vv$. A \textit{\textbf{right dual}} of $X$ is an object $X^*$ together with two morphisms $\ev_X: X\ot X^*\to I$ and $\coev_X: I\to X^*\ot X$ in $\Vv$ called \textbf{\textit{evaluation}} and \textbf{\textit{coevaluation}} satisfying the \textit{zig-zag equations}:
\begin{eqnarray*}
(\ev_X\ot X)(X\ot \coev_X)&=&X,\\
(X^* \ot \ev_X)(\coev_X \ot X^*)&=&X^*.
\end{eqnarray*}
$X$ is called \textit{\textbf{right dualizable}}, if there exist a right dual of $X$ and $\Vv$ is called \textit{\textbf{right rigid}}, if all objects is right dualizable. Similarly, given an object $X$, its \textit{\textbf{left dual}} ${}^*X$ can be defined via $\ev^X: {}^*X\ot X\to I, \coev^X:I\to X\ot {}^*X$ satisfying the analogous equations. $X$ is called \textit{\textbf{left dualizable}}, if there exist a left dual of $X$ and $\Vv$ is called \textit{\textbf{left rigid}}, if all objects are left dualizable. 
\end{Definition}
Duals are not unique, however they are unique up to unique isomorphisms in the following sense:

\begin{proposition} \label{unirig}
Let $\Vv$ be a monoidal category and $X$ an object in $\Vv$. Let $(X^*,\ev_X,\coev_X)$ and $(X',\ev'_X,\coev'_X)$ be two right duals of $X$, then there exists a unique isomorphism $\phi:X^*\to X'$ such that $\ev'_X(X\ot \phi)=\ev_X$ and $\coev'_X=(\phi\ot X)\coev_X$.
\end{proposition}

\begin{remark}
Duals are also compatible with the tensor product: given $X,Y \in \Vv$ with right duals $X^*, Y^*$, then $(X\ot Y)$ is right dualizable with right dual $(Y^*\ot X^*)$ via
\begin{eqnarray*}
\ev_{X\ot Y}&:=&\ev_X(X\ot \ev_Y \ot X^*): X\ot Y\ot Y^*\ot X^* \to I, \\
\coev_{X\ot Y}&:=&(Y^*\ot \coev_X\ot Y)\coev_{Y}: I\to Y^*\ot X^* \ot X\ot Y.
\end{eqnarray*}
Furthermore, in symmetric monoidal categories one need not to differentiate between left and right duals, since any right dual $(X^*,\ev_X,\coev_X)$ of $X$ is also a left dual via $\ev_X\sigma$ and $\sigma \coev_X$.
\end{remark}

\begin{Definition}
Let $\Vv$ be a symmetric monoidal category. $X$ is called \textit{\textbf{dualizable}} if has a left/right dual. $\Vv$ is called \textit{\textbf{rigid}}, if every object is dualizable.
\end{Definition}

Given a morphisms $f:A\to B$ between dualizable objects $A$ and $B$ we consider the morphisms $f^*: B^*\to A^*$ defined by 
$$f^*:=(A^*\ot \ev_B)(A^*\ot f\ot B^*)(\coev_A \ot B^*).$$
\[
\begin{tikzpicture}
\node (1) at (0,0) {$B^*$};
\node (2) at (0,-2) {$A^*$};
\node[morphism] (f) at (0,-1) {$f^*$};
\draw (1) to [out=-90,in=90] (f.north);
\draw (f.south) to [out=-90,in=90] (2);
\node (1) at (1,-1) {$=$};
\end{tikzpicture} 
\begin{tikzpicture}
\node (1) at (0,0) {$B^*$};
\node (2) at (-2,-2) {$A^*$};
\node (0) at (-1.5,0) {};
\node (5) at (-0.5,-2) {};
\node[morphism] (f) at (-1,-1) {$f$};
\draw (1) to [out=-90,in=0] (5.center);
\draw (f.south) to [out=-90,in=180] (5.center);
\draw (0.center) to [out=180, in=90] (2);
\draw (0.center) to [out=0, in=90] (f.north);
\end{tikzpicture} 
\]
One easily checks that this morphism satisfies:
\begin{eqnarray*}
(f^*\ot B)\coev_B&=&(A^*\ot f)\coev_{A}\\
\ev_A (A\ot f^*)&=& \ev_B(f\ot B^*)\\
\end{eqnarray*}

The following is again well known and therefore stated without proof.
\begin{proposition} \label{natrig}
Let $(\Vv,\ot, I)$, $(\Ww,\boxtimes, J)$ be two strict monoidal categories, $(F,\phi,\phi_0),(G,\psi,\psi_0): \Vv\to \Ww$ be two monoidal functors and $\alpha: F\Rightarrow G$ a monoidal natural transformation. If $\Vv$ is right rigid, then $\alpha$ is a natural isomorphism.
\end{proposition}

\subsection{The relative fundamental groupoid}

\begin{Definition}
Let $(X,Y)$ be a pair of topological spaces. Define $\Pi_{1}(X,Y)$ the \textit{\textbf{fundamental groupoid}} of $(X,Y)$ as the following category:
\begin{itemize}
\item Objects are given by $y\in Y$.
\item Given $y,z \in Y$, the Hom-sets are given by 
$$\Pi_1(X,Y)(y,z):=\left\{[f]| f:[0,1]\to X, f(0)=y,f(1)=z\right\},$$
 where $[f]=[g]$ iff they are $f$ and $g$ are homotopic relativ to the endpoints. 
\item The composition of $[f]$ and $[g]$ is given by $[g]\circ [f]=[f\cdot g]$, where 
\begin{eqnarray*}
(f\cdot g)(t):=\begin{cases}
f(2t) & t\in [0,\frac{1}{2}]\\
g(2t-1) & t\in [\frac{1}{2},1]
\end{cases}
\end{eqnarray*} 
\item The units are given by the constant paths $1_y$ for $y\in Y$.
\item The inverse of $[f]$ is given by $[\bar f]$, where 
$$\bar f(t):= f(1-t)$$.
\end{itemize}
Denote the homsets $\Hom_{\Pi_1(X,Y)}(y,z)=:\Pi_1(y,z)$, if there is no ambiguity about the pair $(X,Y)$.
\end{Definition}

One checks easily that the fundamental groupoid is in fact a groupoid, meaning that $[f]\circ[\bar f]=1_z$ and $[\bar f]\circ[f]=1_y$ for all $f\in \Pi_1(y,z)$. In the following we write $\alpha: x\to y$ or $\beta:y\to z$ for elements in the fundamental groupoid and $\alpha\beta: x\to z$ for the composition $\beta\circ \alpha \in \Pi_1(x,z)$.

\section{The category of $(X,Y)$-cobordisms}
To provide a categorical definition of $(X,Y)$-HQFTs it is helpful to put homotopy conditions already in the definition of the source category and not as a condition on the functor. So throughout the following let $X$ be a homotopy 1-type, $Y\subset X$ and $d\geq0$ be a natural number. By manifold we will always mean a differentiable manifold and by closed we understand it to be compact and with empty boundary.

\subsection{$(X,Y)$-cobordisms}
\begin{Definition}
A topological space $Z$ is \textit{\textbf{pointed}} if every connected component of $Z$ is provided with a base point. Denote the set of base points of $Z$ by $Z_\bullet$.
\end{Definition}

\begin{Definition} 
 A $d$-dimensional \textit{\textbf{$(X,Y)$-manifold}} is a pair $(M,[g])$, where $(M,M_\bullet)$ is a pointed, oriented, $d$-dimensional manifold and $[g] \in [(M, M_\bullet), (X,Y)]_\bullet$ a homotopy class of maps relative to the base points. Particularly, $M$ is allowed to be the empty set and $g$ the empty function. Denote it by $\emptyset_d$ to avoid ambiguities. $g$ is called the \textit{\textbf{characteristic map}}. An \textit{\textbf{$(X,Y)$-homeomorphism of $(X,Y)$-manifolds}} between $(M,[g])$ and $(M',[g'])$ is an orientation preserving homeomorphism $f:(M,M_\bullet)\to (M',M'_\bullet)$ such that $[g]=[g'f]$ in $[(M, M_\bullet), (X,Y)]_\bullet$. 
\end{Definition}

\begin{Definition}
A $(d+1)$-dimensional \textit{\textbf{$(X,Y)$-cobordism}} is a tuple $(W,M_0,M_1,f)$, together with a homotopy class of maps $[g]\in [(W,\partial W), (X,Y)]_{\partial W}$ relative to the boundary $\partial W$. $M_0$ and $M_1$ are disjoint, pointed, closed, oriented, $d$-dimensional $(X,Y)$-manifolds with characteristic maps $[g_0]\in[(M_0,M_{0\bullet}),(X,Y)]_\bullet,$ $[g_1]\in[(M_1,M_{1\bullet}),(X,Y)]_\bullet$ and $W$ is a compact, oriented, $(d+1)$-dimensional manifold with $f=(f_0\coprod f_1): \partial W\to -M_0 \coprod M_1$ an orientation preserving homeomorphisms, where $- M_0$ denotes the manifold $M_0$ with opposite orientation. Further for any representative $g:M\to X$ we have that $[g\circ f_0^{-1}]=[g_0]$ and $[g\circ f_1^{-1}]=[g_1]$ in $[(M_0,M_{0\bullet}),(X,Y)]_\bullet$ and $[(M_0,M_{0\bullet}),(X,Y)]_\bullet$ respectively. An \textit{\textbf{$(X,Y)$-homeomorphism of $(X,Y)$-cobordisms}} is a triple $(F,\phi_0,\phi_1):(W,M_0,M_1,f)\to (W',M'_0,M'_1,f')$, where $F:W\to W'$ is an orientation preserving homeomorphism and $\phi_i: M_i\to M'_i$ are $(X,Y)$-homeomorphisms of $(X,Y)$-manifolds for $i\in \{0,1\}$, such that $(\phi_0\coprod \phi_1)\circ f\circ F|_{\partial W}=f'$ and $[g]=[g'F]$ in $[(W,\partial W), (X,Y)]_\bullet$ .
\end{Definition}

\subsection{Glueing of $(X,Y)$-cobordisms}
Since we work with a homotopy 1-type $X$ glueing of $(X,Y)$-cobordisms poses no problems: Let $(M_0,[g_0]), (M_1,[g_1])$ and $(N,[h])$ be $(X,Y)$-manifolds and fix representatives $g_0:M_0\to X$ and $g_1:M_1\to X$. Let $(W,M_0,N,f)$ and 
$(W',N,M_1,f')$ be $(X,Y)$-cobordisms together with two representatives $g:W\to X$ and $g':W'\to X$. Since $g\circ f^{-1}_1$ and $g'\circ f'^{-1}$ are both representatives of the class $[h]\in [(N,N_\bullet),(X,Y)]_\bullet$ we have a homotopy $T:N\times [0,1] \to X$ such that $T(n,0)=(g\circ f_1^{-1})(n),$ $T(n,1)=(g'\circ (f'_0)^{-1})(n)$ for all $n\in N$ and $T(n_0,t)=y_0$ for base points $n_0 \in N_\bullet$ and $y_0 \in Y$, hence we can glue $W$ to $N\times [0,1]$ from the left via $\pi_0^{-1}f_1$ and glue $W'$ from the right via $f'^{-1}_0\pi_1$, where $\pi=\pi_0\coprod \pi_1:N\times \{0,1\}\to -N\coprod N$. We define this new $(X,Y)$-cobordism by $(W\circ W',M_0,M_1, (f_0,f'_1))$ and the homotopy class of $g\circ g':=g \coprod_{\pi_0^{-1}f_1} T \coprod_{f'^{-1}_0\pi_1} g'$. We call $(W\circ W',[g\circ g'])$ the \textit{glueing} of $W$ and $W'$.

\subsection{The category of $(X,Y)$-cobordisms}

$(X,Y)$-cobordisms form the following category $\Cob^{X,Y}_{d+1}$:
\begin{itemize}
\item Objects are oriented, $d$-dimensional compact $(X,Y)$-manifolds.
\item Morphisms are homeomorphism classes of $d+1$-dimensional $(X,Y)$-cobordism.
\item Composition is given by gluing of $(X,Y)$-cobordisms. Given $(W,M_0,N,f), (W',N,M_1,f')$ their composition is defined to be $(W\circ W',[g\circ g'])$.
\item Identities $\Id_{(M,g)}$  on a given $(M,g)$ an oriented, $d$-dimensional $(X,Y)$-manifold are given by the homeomorphism class of the $(X,Y)$-bordism $(M\times[0,1],M\times 0, M\times 1, \pi|_{M\times \{0,1\}})$ with homotopy class $[g\circ\pi]$ the composition of $g$ and projection on $M$. This might be called the \textit{\textbf{cylinder}} over $(M,g)$.
\end{itemize}
Further, there is a monoidal structure on $\Cob^{X,Y}_{d+1}$:
\begin{itemize}
\item The Tensor product is given by the disjoint union. Given two $(X,Y)$-manifolds $(M,[g]), (N,[g'])$ consider $(M\coprod N, [g\coprod g'])$ and similarly for the disjoint union of $(X,Y)$-bordisms, satisfying associativity constraints.
\item The Monoidal unit is the empty $(X,Y)$-manifold $\emptyset_{d}$, satisfying unitality constraints.
\item With the braiding $\sigma_{MN}: (M\coprod N, [g\coprod g'])\to (N\coprod M, [g'\coprod g])$, with $\sigma_{NM}\sigma_{MN}=\Id_{M\coprod N}$, $\Cob^{X,Y}_{d+1}$ satisfies the conditions of a symmetry of monoidal categories.
\end{itemize}
This symmetric monoidal category is called the \textit{\textbf{category of $(d+1)$-dimensional $(X,Y)$-cobordisms}}.
 
\begin{proposition} \label{rig}
The category $\Cob^{X,Y}_{d+1}$ is rigid. 
\end{proposition}

This is easily seen by considering the cylinders $M\times [0,1]$ over oriented $d$-dimensional compact $(X,Y)$-manifolds $(M,[g])$ and interpreting them as morphisms $C: \emptyset\to -M \coprod M$ as well as $E: M\coprod -M \to \emptyset$, where $-M$ is the manifold $M$ with opposite orientation, by setting their homotopy class $[g\circ\pi]$ to be the composition of $g$ and projection on $M$. $E$ and $C$ statisfy the zig-zag-equations exhibiting $(-M,[g])$ as the dual of $M$. Furthermore, for any bordism $(W,M_0, M_1, f_0\coprod f_1)$ the cobordism $-W=(-W, -M_1, -M_0, f_1\coprod f_0)$ with the same characteristic map is given by 
$$W^*= (-M_0 \ot \ev_{M_1})(-M_0 \ot W \ot -M_1)(\coev_{M_0}\ot -M_1)$$.

\section{$(X,Y)$-HQFTs and their classification in dimension 1}
In this formalism $d$-dimensional $(X,Y)$-HQFTs can now be defined quite easily with a variable target category. Please note, that there is some ambiguity in the term \textit{target}! In the theory of TQFT it refers to the domain of the functor, whereas in \cite{Tu1} it refers to the space $X$. We make this nomenclature precise. Through this section let $d>0$.
\begin{Definition}
Let $(\Vv,\ot,I,\sigma)$ be a (strict) symmetric, monoidal category. A $d$-dimensional \textit{\textbf{Homotopy Quantum Field Theory}} with \textit{\textbf{target pair of spaces}} $(X,Y)$, or shorter $(X,Y)$-\textit{\textbf{HQFT}} is a symmetric monoidal functor 
$$\Zz:  \Cob^{X,Y}_{d} \to \Vv.$$
$\Vv$ is called its \textit{\textbf{target category}}. 
\end{Definition}


\begin{Definition}
The symmetric monoidal functor category $\sf Fun^{\ot, s}(\Cob^{X,Y}_d,\Vv)$ of $d$-dimensional $(X,Y)$-HQFTs and target space pair $(X,Y)$ and target category $\Vv$ is denoted by $\Qq_d(X,Y,\Vv)$. A morphism between two $(X,Y)$-HQFTs $\Zz$ and $\Zz'$ is hence a monoidal natural transformation $\rho: \Zz \Rightarrow \Zz'$.
\end{Definition}

\begin{remark}
Since $\Cob^{X,Y}_d$ is rigid, all natural transformations between $(X,Y)$-HQFTs are isomorphisms due to \ref{natrig}. This means $\Qq_d(X,Y,\Vv)$ is a groupoid. \\
A natural question to ask is how $(X,Y)$-HQFT's behave under change of target space. Given an $(X,Y)$-HQFT $\Zz$ and a map $f:X\to X'$, then the HQFT can be pushed forward to an $(X',f(Y))$-HQFT via $f_*: \Qq_d(X,Y,\Vv)\to \Qq_d(X',Y,\Vv)$ by composing the characteristic map with $f$. Similarly if we are interested in changing the target category and can always compose with a symmetric monoidal functor $F: \Vv\to \Vv'$ to get a push forward $F_*: \Qq_d(X,Y,\Vv)\to \Qq_d(X,Y,\Vv')$. 
\end{remark}

Although the definition of $\Cob_d^{X,Y}$ and $d$-dimensional $(X,Y)$-HQFTs makes sense for any target pair with $X$ a homotopy 1-type, we impose some light connectivity condition on $(X,Y)$, namely that every connected component of $X$ has non-empty intersection with $Y$. The reason for that is, that we otherwise might have morphisms in $\Cob_d^{X,Y}$ from the empty manifold to the empty manifold, whose characteristic map maps into a component without any point of $Y$ in it. Such a morphism cannot factor through any object different from the empty manifold. Therefore we might not represent is by composition of simple generators. It is therefore reasonable to assume that any connected component of $X$ has a point in $Y$. 

\begin{Definition}
A pair of topological spaces $(X,Y)$ is called \textit{\textbf{admissible}}, if $X$ is a homotopy 1-type and for any connected component $C$ of $X$ we have $C\cap Y\neq \emptyset$.
\end{Definition}

\begin{Definition}
Let $\Gg$ be a category and $\Vv$ a monoidal category. A \textit{\textbf{dualizable representation}} of $\Gg$ in $\Vv$ is a functor $F:\Gg \to \Vv$ such that $Fx$ is dualizable in $\Vv$ for all objects $x$ of $\Gg$. A \textit{\textbf{morphism of dualizable representations}} is a natural transformation $\rho: F\to F'$ such that 
\begin{eqnarray*}
\ev_{F'x}(\rho_x\ot \rho_{x^*})&=&\ev_{Fx}\\
(\rho_{x^*}\ot \rho_x)\coev_{Fx}&=&\coev_{F'x}\\
\end{eqnarray*}
Denote the category of dualizable representations and morphisms of dualizable representations as $\Fun^*(\Gg,\Vv)$.
\end{Definition}

\begin{theorem}
Let $(\Vv,\ot,I,\sigma)$ be a (strict) symmetric monoidal category. Let $(X,Y)$ be an admissible pair of spaces and denote the relative fundamental groupoid of the pair $(X,Y)$ by $\Gg=\Pi_1(X,Y)$. The category $\Qq_1(X,Y,\Vv)$ of 1-dimensional $(X,Y)$-HQFTs 
 is equivalent to the category $\Fun^*(\Gg,\Vv)$ of dualizable representations.
\end{theorem}

\begin{proof}
The objects in $\Cob^{X,Y}_1$ are oriented, 0-dimensional, closed manifolds, hence consist of a finite, disjoint union of oriented points, together with a homotopy class of maps to $Y$ fixing basepoints. One therefore can identify them with points in $Y$ together with an orientation $\in \{+,-\}$. The duality between $+y$ and $-y$ is exhibited by $\ev_y: +y\coprod -y\to \emptyset$ and $\coev_y: \emptyset \to -y \coprod +y$, whose characteristic maps are the constant maps on $y$. Consider $y,z\in Y$ and a paths $g:[0,1]\to X$ from $y$ to $z$, then there exists a morphism from $+y$ to $+z$ given by the interval and characteristic class $[g]$, that we might also denote by $[g]$. We claim that any morphisms in $\Cob^{X,Y}_1$ can be glued from $\sigma, \ev_{y},\coev_{y}$ for any $y\in Y$ and $[g]$ for any $g:y\to z \in \Gg$.\\
\[
\begin{tikzpicture}
\node (a) at (-2,0) {$+y$};
\node (b) at (-1,0) {$+z$};
\node (c) at (-2,-1) {$+z$};
\node (d) at (-1,-1) {$+y$};
\node (1) at (0,0) {$+y$};
\node (2) at (1,0) {$-y$};
\node (3) at (2,-1) {$-y$};
\node (4) at (3,-1) {$+y$};
\node (5) at (4,0) {$+y$};
\node (6) at (5,-1) {$+z$};
\node (7) at (4.5,-0.2) {$g$};
\draw (1) to [out=-90,in=-90] (2);
\draw (3) to [out=90,in=90] (4);
\draw (5) to [out=-90,in=90]  (6);
\draw (a) to  [out=-90,in=90] (d);
\draw (c) to  [out=90,in=-90] (b);
\end{tikzpicture} 
\]

First of all, any morphisms $-y\to -z$ with characteristic map $[g]$ can be represented as $[\bar g]^*=(\Id_{-z}\ot \ev_{+y})(\Id_{-z} \ot [\bar g] \ot \Id_{-y})(\coev_{+z} \ot \Id_{-y})$.
\[
\begin{tikzpicture}
\node (1) at (0,0) {$-y$};
\node (2) at (-2,-2) {$-z$};
\node (0) at (-1.5,0) {};
\node (5) at (-0.5,-2) {};
\node[morphism] (f) at (-1,-1) {$\bar g$};
\draw (1) to [out=-90,in=0] (5.center);
\draw (f.south) to [out=-90,in=180] (5.center);
\draw (0.center) to [out=180, in=90] (2);
\draw (0.center) to [out=0, in=90] (f.north);
\end{tikzpicture} 
\]
Furthermore, any morphism $\emptyset \to y \ot -z$ with characteristic class $[g]$ can be represented as $(+y \ot [\bar g])\coev_z$ and similarly for  $-y \ot +z \to \emptyset$.  Pre- or postcomposing with $\sigma$ yield all morphisms $\emptyset \to -y \ot +z$ and $+y \ot -z \to \emptyset$. Therefore any morphisms with underlying 1-manifold an interval is a composition our generators. For any morphisms with underlying 1-manifold a circle, we can assume by connectivity that the caracteristic maps hits at least on point in $Y$, say $y$. Any circle therefore factors through $\ev_{y}\sigma (\Id_{-y}\ot [g] )\coev_y$, where $g:y\to y$ in $\Gg$. 
Since morphisms in $\Cob^{X,Y}_1$ are finite unions of interval and circles, all morphisms are generated by $\sigma, \ev_{y},\coev_{y}$ for any $y\in Y$ and $[g]$ for any $g:y\to z \in \Gg$.\\
Now given a 1-dimensional $(X,Y)$-HQFT $\Zz$, the functor determines objects $\Zz(+y)$ in $\Vv$ with choice of duals $\Zz(-y)$, $\Zz(\coev_y), \Zz(\ev_y)$ together morphisms 
$$\Zz([g]): \Zz(+y)\to \Zz(+z)$$ 
in $\Vv$ for $g:y\to z$. One checks easily that these form an action of the groupoid $\Gg$ on the objects $\Zz(y)$. \\
Conversely, any such data is sufficient to define a 1-dimensional $(X,Y)$-HQFTs, since $\Zz$ is determined by the value on the morphisms $\ev_y,\coev_y$ and $[g]$.\\
Let now $\rho: \Zz \Rightarrow \Zz'$ be a natural isomorphism between two 1-dimensional $(X,Y)$-HQFTs, hence we get isomorphisms $\rho_{\pm y}: \Zz(\pm y)\to \Zz'( \pm y)$ such that $\ev'_y(\rho_{+y}\ot \rho_{-y})=\ev_y$ and $(\rho_{-y}\ot \rho_{+y})\coev_y=\coev'_y$ and $\rho_z \Zz([g])= \Zz([g])\rho_y$, hence this is the same data than a morphism of dualizable representations.
\end{proof}

\section{Crossed loop Frobenius $(\Gg,\Vv)$-categories}
In this section the main notion in order to classify 2-dimensional $(X,Y)$-HQFTs is defined. These structures should specify to $G$-graded Frobenius algebras 
$$L=\bigoplus_{\alpha\in G} L_\alpha$$
 in the case where $\Gg=G$ a group and $\Vv=\Vect_\KK$. However, since we do not assume the existence of (co)-products in $\Vv$, consider the $L_\alpha$ themselves as the objects. As one can see in the proof of the main theorem, this is actually much more in the spirit of HQFTs.

\begin{Definition}
Let $\Gg$ be a groupoid and  $\Vv$ a (strict) symmetric monoidal category. A \textbf{\textit{$(\Gg,\Vv)$-category}} is a collection of objects $L_\alpha\in \Vv$ for all $\alpha:x\to y$ in the groupoid $\Gg$, together with two families of morphisms in $\Vv$
\begin{eqnarray*}
m_{\alpha,\beta}: L_{\alpha}\ot L_{\beta} \to L_{\alpha \beta} \qquad \alpha:x\to y, \beta:y\to z, x,y,z \in \Gg,
\end{eqnarray*}
called \textit{\textbf{composition}}/\textit{\textbf{multiplication}}, and
\begin{eqnarray*}
j_{x}: I\to L_{1_x} \qquad x \in \Gg, 
\end{eqnarray*}
called \textit{\textbf{unit}}, satisfying associativity:
\[
\xymatrix{
L_{\alpha}\ot L_{\beta} \ot L_\gamma \ar[d]_{L_\alpha \ot m_{\beta,\gamma}}  \ar[rr]^-{m_ {\alpha,\beta} \ot L_ \gamma} & &L_{\alpha\beta} \ot L_\gamma  \ar[d]^{m_{\alpha\beta,\gamma}} &\\
L_{\alpha}\ot L_{\beta\gamma} \ar[rr]_-{m_{\alpha,\beta\gamma}} & & L_{\alpha\beta\gamma}\\
}
\]
for all composable triple of morphisms $\alpha, \beta, \gamma$ in $\Gg$ and unitality:
\[
\xymatrix{
L_{\alpha}\simeq I \ot L_\alpha \ar[rr]^-{j_{x} \ot L_\alpha}  \ar[rrd]_{L_\alpha} & & L_{1_x} \ot L_{\alpha} \ar[d]^{m_{1_x,\alpha}} & L_{\alpha}\simeq L_\alpha \ot I \ar[rr]^-{L_\alpha \ot j_y}  \ar[rrd]_{L_{\alpha}} & &L_{\alpha} \ot L_{1_y} \ar[d]^{m_{\alpha,1_y}}\\
& &L_{\alpha}, & & &L_{\alpha}  \\
}
\]
 for $\alpha:x\to y$ in $\Gg$. \\
A $(\Gg,\Vv)$-\textit{\textbf{functor}} is a collection of morphisms $F_{\alpha}: L_\alpha \to L'_\alpha$ in $\Vv$ commuting with composition and unit:
\[
\xymatrix{
L_{\alpha}\ot L_{\beta} \ar[r]^{F_\alpha \ot F_\beta}  \ar[d]_{m_{\alpha,\beta}} & L'_{\alpha} \ot L'_\beta  \ar[d]^{m'_{\alpha,\beta}} & I \ar[r]^{j_x} \ar[rd]_{j'_{Fx}} & L_{1_x} \ar[d]^{F_{1_x}}\\
L_{\alpha\beta} \ar[r]_{F_{\alpha\beta}} & L'_{\alpha\beta}, & & L'_{1_{Fx}}. \\
}
\]
Denote the category of $(\Gg,\Vv)$-categories as $(\Gg,\Vv)-\Cat$.
\end{Definition}

\begin{remark}
Let $X$ be a class and $\Gg=X\times X$ the groupoid with objects $X$, that has exactely one arrow $x\to y$ for all $x,y\in X$. Such $(X\times X,\Vv)$-categories are precisely categories enriched over $\Vv$, hence, this notion generalizes enriched categories. Note that generally, for given objects $x,y$ in $\Gg$, there are multiple object $L_\alpha$ in $\Vv$, namely for all $\alpha:x\to y$ in $\Gg$, and not just a collection $L_{xy}=\Hom_L(x,y) \in \Vv$ as for enriched categories. Further, note that the objects $x,y$ play only a supporting role by clarifying, which elements can be composed.
\end{remark}

\begin{example}
Given a field $\KK$ and a groupoid $\Gg$. Define $L_\alpha=\KK \langle l_\alpha \rangle$ to be the 1-dimensional vector spaces generated by $l_\alpha$ and $m_{\alpha,\beta}(l_\alpha\ot l_\beta):=l_{\alpha\beta}$ and $j_x(1):=l_{1_x}$. $(L,m,j)$ is a $(\Gg,\Vect_\KK)$-category in analogue to the group algebra for a group $G$. We may therefore call it \textit{\textbf{groupoid algebra}} of $\Gg$ and denote it by $L=\KK[\Gg]$.
\end{example}

\begin{Definition}
Let $\Gg$ be a groupoid and  $(\Vv,\ot,I,\sigma)$ a (strict) symmetric monoidal category. A \textbf{\textit{$(\Gg,\Vv)$-opcategory}} is a $(\Gg,\Vv^{op})$-category $L$.\\
Explicitely, $L$ consists of a collection of objects $L_\alpha\in \Vv$ for all $\alpha:x\to y$ in the groupoid $\Gg$, together with two families of morphisms in $\Vv$:
\begin{eqnarray*}
\Delta_{\alpha,\beta}: L_{\alpha\beta} \to L_{\alpha} \ot L_{\beta} \qquad \alpha:x\to y, \beta:y\to z, x,y,z \in \Gg,
\end{eqnarray*}
called \textit{\textbf{cocomposition}}/\textit{\textbf{comultiplication}}, and
\begin{eqnarray*}
\nu_{x}: L_{1_x}\to I \qquad x \in \Gg, 
\end{eqnarray*}
called \textit{\textbf{counit}}, satisfying coassociativity:
\[
\xymatrix{
L_{\alpha}\ot L_{\beta} \ot L_\gamma & & L_{\alpha\beta} \ot L_\gamma \ar[ll]_-{\Delta_ {\alpha,\beta} \ot L_ \gamma} \\
L_{\alpha}\ot L_{\beta\gamma}\ar[u]^{L_\alpha \ot \Delta_{\beta,\gamma}} & & L_{\alpha\beta\gamma}\ar[ll]^-{\Delta_{\alpha,\beta\gamma}} \ar[u]_{ \Delta_{\alpha\beta,\gamma}} \\
}
\]
for all composable triples of morphisms $\alpha, \beta, \gamma$ in $\Gg$ and counitality
\[
\xymatrix{
L_{\alpha}\simeq I \ot L_\alpha &  & L_{1_x} \ot L_{\alpha} \ar[ll]_-{\nu_{x}\ot L_{\alpha} } & L_{\alpha}\simeq L_\alpha \ot I  &  &L_{\alpha} \ot L_{1_y} \ar[ll]_-{L_\alpha \ot \nu_y} \\
& &L_{\alpha},  \ar[ull]^{L_\alpha}  \ar[u]_{\Delta_{1_x,\alpha}}  & & &L_{\alpha}. \ar[lul]^{L_{\alpha}}  \ar[u]_{\Delta_{\alpha,1_y}}\\
}
\]
for $\alpha:x\to y$ in $\Gg.$ \\
A $(\Gg,\Vv)$-\textit{\textbf{op-functor}} is a collection $F_{\alpha}: L_\alpha \to L'_\alpha$  of maps in $\Vv$ commuting with cocomposition and counit:
\[
\xymatrix{
L_{\alpha}\ot L_{\beta} \ar[r]^{F_\alpha \ot F_\beta}   & L'_{\alpha} \ot L'_\beta  & I  & L_{1_x}  \ar[l]_{\nu_x} \ar[d]^{F_{1_x}}\\
L_{\alpha\beta}\ar[u]^{\Delta_{\alpha,\beta}} \ar[r]_{F_{\alpha\beta}} & L'_{\alpha\beta}, \ar[u]_{\Delta'_{\alpha,\beta}} & & L'_{1_{Fx}}. \ar[ul]^{\nu'_{Fx}}\\
}
\]
Denote the category of $(\Gg,\Vv)$-opcategories as $(\Gg,\Vv)-\opCat$.
\end{Definition}

Frobenius algebras can be defined by two equivalent definitions the existence of an inner product as done in \cite{Tu1} or the categorical definition stating the existence of a compatible comultiplication. In the following generalizations of both are introduced and it is shown, that these are equivalent.

\begin{Definition} \label{Frob1}
Let $\Gg$ be a groupoid and  $(\Vv,\ot,I,\sigma)$ a (strict) symmetric monoidal category. A \textit{\textbf{Frobenius $(\Gg,\Vv)$-category}} is a 5-tuple consisting of a family $(L_\alpha)_{\alpha \in \Gg}$ of objects in $\Vv$ and four families of morphisms $m,j,\Delta,\nu$ in $\Vv$ such that: 
\begin{itemize}
\item[(FC1)] $(L,m,j)$ is an $(\Gg,\Vv)$-category.
\item[(FC2)] $(L,\Delta, \nu)$ is a $(\Gg,\Vv)$-opcategory.
\item[(FC3)] The \textit{\textbf{Frobenius conditions}} are satisfied:
\end{itemize}
\[
\xymatrix{
& L_{\alpha\beta}\ot L_{\gamma} \ar[dl]_{\Delta_{\alpha,\beta} \ot L_\gamma}  \ar[dr]^{m_{\alpha\beta, \gamma} }& & L_{\alpha}\ot L_{\beta\gamma} \ar[dl]_{m_{\alpha,\beta\gamma}} \ar[dr]^{L_\alpha \ot \Delta_{\beta,\gamma}} &\\
L_{\alpha}\ot L_\beta \ot L_\gamma  \ar[dr]_{L_\alpha\ot m_{\beta,\gamma}}& & L_{\alpha\beta\gamma} \ar[dl]^{\Delta_{\alpha,\beta\gamma}}  \ar[dr]_{\Delta_{\alpha\beta,\gamma}} & &L_{\alpha}\ot L_\beta \ot L_\gamma \ar[dl]^{m_{\alpha,\beta} \ot L_\gamma} \\
& L_{\alpha}\ot L_{\beta\gamma} & & L_{\alpha\beta}\ot L_{\gamma}. &
}
  \] 
A \textit{\textbf{Frobenius $(\Gg,\Vv)$-functor}} is a $(\Gg,\Vv)$-functor as well as a $(\Gg,\Vv)$-opfunctor. Denote the category of Frobenius $(\Gg,\Vv)$-categories and -functors by $\Frob^{\Gg}(\Vv)$.
We call $L$ \textit{\textbf{symmetric}}, if for all $\alpha:x\to y$
$$\nu_xm_{\alpha, \alpha^{-1}}=\nu_ym_{\alpha^{-1},\alpha}\sigma: L_{\alpha} \ot L_{\alpha^{-1}}\to I.$$ 
\end{Definition}

Equivalently we can define Frobenius $(\Gg,\Vv)$-categories by defining an inner product: 
\begin{Definition} \label{Frob2}
Let $L$ be a  $(\Gg,\Vv)$-category. A family of morphisms $\eta_{\alpha}: L_\alpha \ot L_{\alpha^{-1}}\to I$ is called an \textit{\textbf{inner product}} or   \textit{\textbf{evaluation}} on $L$ if: 
\begin{itemize}
\item[(IP1)] it is \textit{\textbf{non-degenerate}}, that is, if there exists a family $\coev_{\alpha}: I\to L_{\alpha^{-1}}\ot L_\alpha$ called \textit{\textbf{coevaluation}}, such that
\begin{eqnarray*}
(\eta_\alpha\ot L_\alpha)(L_\alpha \ot \coev_{\alpha})=L_\alpha,\\
(L_{\alpha^{-1}}\ot \eta_\alpha)(\coev_{\alpha}\ot L_{\alpha^{-1}})=L_{\alpha^{-1}}
\end{eqnarray*}
 and 
\item[(IP2)] it is compatibility with the multiplication:
\[
\xymatrix{
L_{\alpha}\ot L_{\beta} \ot L_\gamma \ar[d]_{L_\alpha \ot m_{\beta,\gamma}}  \ar[rr]^{m_ {\alpha,\beta} \ot L_ \gamma} & & L_{\alpha\beta} \ot L_\gamma  \ar[d]^{\eta_{\alpha\beta}} &\\
L_{\alpha}\ot L_{\beta\gamma} \ar[rr]_{\eta_\alpha} & & I, \\
}
\]
\end{itemize}
where $\gamma=(\alpha\beta)^{-1}$, hence $\alpha^{-1}=\beta\gamma.$
The inner product is called \textit{\textbf{symmetric}} if 
$$\eta_\alpha= \eta_{\alpha^{-1}}\sigma: L_\alpha\ot L_{\alpha^{-1}}\to I.$$ 
A \textit{\textbf{Frobenius $(\Gg,\Vv)$-category}} is a $(\Gg,\Vv)$-category $L$ with an inner product. If the inner product is symmetric, $L$ is called a \textit{\textbf{symmetric}} Frobenius $(\Gg,\Vv)$-category.
\end{Definition}

The first conditions states that $L_{\alpha^{-1}}$ it the (right) dual of $L_\alpha$ in $\Vv$ via $\eta_{\alpha}$ and $\coev_\alpha$. The second connects this duality with the $(\Gg,\Vv)$-category  structure on.

\begin{example}
\begin{enumerate}
\item If $\Gg=1$ is the trivial group, with one morphism then every $(1,\Vv)$-category is a Frobenius monoid in $\Vv$.
\item If $\Gg=G$ is a group and $\Vv=\Vect_\KK$, then any $(G,\Vect_\KK)$-category $L$ is equivalent to a Frobenius $G$-algebra $L'$ as defined in \cite{Tu1} given by 
$$L'=\bigoplus_{\alpha \in G} L_\alpha,$$
 with $x\cdot y:= m_{\alpha,\beta}(x\ot y)$ for $x\in L_\alpha, y\in L_\beta$ with $\alpha,\beta \in G$ and $1_{L'}=j_\epsilon(1)\in L_\epsilon$ for $\epsilon \in G$ the neutral element. 
\item The groupoid algebra $L=\KK[\Gg]$ has a $(\Gg,\Vect_\KK)$-opcategory structure via $\Delta_{\alpha,\beta}(l_{\alpha\beta})=l_\alpha \ot l_\beta$ and $\nu(l_{1_x})=1$. Furthermore, the Frobenius conditions are satisfied and therefore $L=\KK[\Gg]$ is a Frobenius $(\Gg,\Vv)$-category, that is symmetric.
\end{enumerate}
\end{example}

\begin{lemma}
Given an inner product, the corresponding coevaluation is unique. Further a Frobenius $(\Gg,\Vv)$-category is symmetric, if and only if  $$\sigma \coev_\alpha = \coev_{\alpha^{-1}}.$$
\end{lemma}

\begin{proof}
The first statement was shown in \ref{unirig}. The second statement follows from the fact that $\eta_X\sigma$ and $\sigma\coev_X$ satisfy the zig-zag-equations. The uniqueness of $\eta$ and $\coev$ shows that $\sigma \coev_\alpha = \coev_{\alpha^{-1}}$, if and only iff $\eta_\alpha \sigma = \eta_{\alpha^{-1}}.$
\end{proof}

\begin{lemma}
Let $L$ be a Frobenius $(\Gg,\Vv)$-category then
\begin{eqnarray}
\eqlabel{lm1} m_{\alpha,\beta} &=&(L_{\alpha\beta} \ot \eta_{(\alpha\beta)^{-1}})(L_{\alpha\beta} \ot L_{(\alpha\beta)^{-1}} \ot m_{\alpha, \beta})(\coev_{(\alpha\beta)^{-1}} \ot L_\alpha \ot L_\beta)\\ 
\eqlabel{lm2} &=&(L_{\alpha\beta} \ot \eta_{\beta^{-1}})(L_{\alpha\beta} \ot m_{(\alpha\beta)^{-1}, \alpha}\ot L_\beta)(\coev_{(\alpha\beta)^{-1}} \ot L_\alpha \ot L_\beta)\\
\eqlabel{lm3} &=&(\eta_{\alpha\beta}\ot L_{\alpha\beta})(m_{\alpha, \beta}\ot L_{(\alpha\beta)^{-1}} \ot L_{\alpha\beta})( L_\alpha \ot L_\beta \ot \coev_{\alpha\beta})\\
\eqlabel{lm4}&=&(\eta_{\alpha}\ot L_{\alpha\beta})(L_{\alpha}\ot m_{\beta, (\alpha\beta)^{-1}} \ot L_{\alpha\beta})(L_\alpha \ot L_\beta \ot \coev_{\alpha\beta})
\end{eqnarray}
\end{lemma}

\begin{proof}
This proof is an easy deduction of the axiom for the coevaluation. By commutativity of tensor product and composition we can push $\coev$ and $\eta$ together in \equref{lm1} to get
$$(L_{\alpha\beta} \ot \eta_{(\alpha\beta)^{-1}})(\coev_{(\alpha\beta)^{-1}}\ot L_{\alpha\beta}) m_{\alpha, \beta}=m_{\alpha,\beta}.$$
\equref{lm2} follows from \equref{lm1} by compability of the multiplication and inner product. \equref{lm3}, \equref{lm4} can be deduced anagolously. 
\end{proof}

\begin{proposition}
Both definitions (\ref{Frob1} \& \ref{Frob2}) of (symmetric) Frobenius $(\Gg,\Vv)$-categories are equivalent.
\end{proposition}

\begin{proof}
Let $L$ be a Frobenius $(\Gg,\Vv)$-category in the sense of \ref{Frob1}. We will show that there is an inner product on $L$. For $\alpha:x\to y$ define $\eta_\alpha:=\nu_x m_{\alpha,\alpha^{-1}}$. The proof works in analogue to the one-object case, but making sure we respect codomains and domains of arrows. 
\begin{itemize}
\item $\eta_\alpha$ is symmetric if and only if $L$ is a symmetric Frobenius $(\Gg,\Vv)$-category.
\item Non-degeneracy condition: Define $\coev_\alpha:=\Delta_{\alpha^{-1},\alpha}j_y$. Then 
\begin{eqnarray*}
(L_ {\alpha^{-1}} \ot \eta_{\alpha})(\coev_{\alpha}\ot L_{\alpha^{-1}})&=&
(L_{\alpha^{-1}} \ot \nu_x)(L_\alpha \ot m_{\alpha, \alpha{-1}} )(\Delta_{\alpha^{-1},\alpha} \ot L_\alpha) (j_y \ot L_\alpha)\\
&=&(L_{\alpha^{-1}} \ot \nu_x)\Delta_{\alpha^{-1},1_x} m_{1_y,\alpha^{-1}}(j_y \ot L_{\alpha^{-1}})=L_{\alpha^{-1}},\\
(\eta_{\alpha} \ot L_\alpha)(L_\alpha \ot \coev_{\alpha})&=&(\nu_x \ot L_\alpha)( m_{\alpha,\alpha^{-1}}\ot L_\alpha)(L_\alpha \ot \Delta_{\alpha^{-1},\alpha})(L_\alpha\ot  j_y)\\
&=&(\nu_x \ot L_\alpha)\Delta_{1_x,\alpha}m_{\alpha,1_y}(L_\alpha\ot  j_y)=L_{\alpha},
\end{eqnarray*}
where we used the Frobenius conditions as well as (co)-unitality.
\item For the compatibility with the multiplication we can look at the following commuting diagram:
\[
\xymatrix{
L_{\alpha}\ot L_{\beta} \ot L_\gamma \ar[d]_{L_\alpha \ot m_{\beta,\gamma}}  \ar[rr]^{m_ {\alpha,\beta} \ot L_ \gamma} & & L_{\alpha\beta} \ot L_\gamma  \ar[d]^{m_{\alpha\beta,\gamma}} \ar@/^2.0pc/[ddr]^{\eta_{\alpha\beta}} &\\
L_{\alpha}\ot L_{\beta\gamma} \ar[rr]_{m_{\alpha,\beta\gamma}} \ar@/_2.0pc/[drrr]_{\eta_{\alpha} } & & L_{\alpha\beta\gamma}=L_{1_x} \ar[rd]^{\nu_x} & \\
& && I.
}
\]
The square commutes by associativity of the multiplication and the two triagles commute per definitionem.
\end{itemize}
Therefore $\eta$ is an inner product on the $(\Gg,\Vv)$-category $L$, making it is Frobenius in terms of \ref{Frob2}. Further $L$ is symmetric if the inner product is symmetric.\\
Now let $L$ be a Frobenius $(\Gg,\Vv)$-category in the sense of \ref{Frob2}. We need to construct a comultiplication and a counit. We define 
\begin{eqnarray*}
\Delta_{\alpha,\beta}&:=&(L_\alpha\ot m_{\alpha^{-1},\alpha\beta})(\coev_{\alpha^{-1}} \ot L_{\alpha\beta}): L_{\alpha\beta}\to L_\alpha \ot L_\beta,\\
\nu_x&:=&\eta_{1_x} (j_x\ot L_{1_x}):L_{1_x}\to I.
\end{eqnarray*}
First note that $\Delta_{\alpha,\beta}=(m_{\alpha\beta,\beta^{-1}} \ot L_\beta)( L_{\alpha\beta}\ot \coev_{\beta})$ by using \equref{lm2} and \equref{lm3}:
\begin{eqnarray*}
\Delta_{\alpha,\beta}&=&(L_\alpha\ot m_{\alpha^{-1},\alpha\beta})(\coev_{\alpha^{-1}} \ot L_{\alpha\beta}) \\
&=&(L_\alpha \ot \eta_{\beta}\ot L_{\beta})(L_\alpha \ot m_{\alpha^{-1}, \alpha\beta}\ot L_{\beta^{-1}} \ot L_{\beta})(L_\alpha \ot L_{\alpha^{-1}} \ot L_{\alpha\beta} \ot \coev_{\beta})(\coev_{\alpha^{-1}} \ot L_{\alpha\beta})\\
&=&(L_\alpha \ot \eta_{\beta}\ot L_{\beta})(L_\alpha \ot m_{\alpha^{-1}, \alpha\beta}\ot L_{\beta^{-1}} \ot L_{\beta})(\coev_\alpha \ot L_{\alpha\beta} \ot L_{\beta^{-1}} \ot L_{\beta} )(L_{\alpha\beta}\ot \coev_{\beta})\\
&=&(m_{\alpha\beta,\beta^{-1}} \ot L_\beta)( L_{\alpha\beta}\ot \coev_{\beta}).
\end{eqnarray*}
Using this identity, we show coassociativity:
\begin{eqnarray*}
(\Delta_{\alpha,\beta}\ot L_\gamma) \Delta_{\alpha\beta,\gamma}
&=&(L_\alpha\ot m_{\alpha^{-1},\alpha\beta}\ot L_\gamma)(\coev_{\alpha^{-1}} \ot L_{\alpha\beta}\ot L_\gamma)(m_{\alpha\beta\gamma,\gamma^{-1}} \ot L_\gamma)( L_{\alpha\beta\gamma}\ot \coev_{\gamma})\\
&=&(L_\alpha\ot m_{\alpha^{-1},\alpha\beta}\ot L_\gamma)(L_\alpha \ot L_{\alpha^{-1}} \ot m_{\alpha\beta\gamma,\gamma^{-1}} \ot L_\gamma)( \coev_{\alpha^{-1}} \ot L_{\alpha\beta\gamma}\ot \coev_{\gamma})\\
&=&(L_\alpha \ot m_{\beta\gamma,\gamma^{-1}}\ot L_\gamma)(L_\alpha \ot m_{\alpha^{-1},\alpha\beta\gamma} \ot L_\gamma^{-1} \ot L_\gamma)( \coev_{\alpha^{-1}} \ot L_{\alpha\beta\gamma}\ot \coev_{\gamma})\\
&=&(L_\alpha \ot m_{\beta\gamma,\gamma^{-1}}\ot L_\gamma)(L_\alpha \ot L_{\beta\gamma} \ot \coev_{\gamma})(L_\alpha \ot m_{\alpha^{-1},\alpha\beta\gamma})(\coev_{\alpha^{-1}} \ot L_{\alpha\beta\gamma})\\
&=&(L_\alpha \ot \Delta_{\beta,\gamma}) \Delta_{\alpha,\beta\gamma},
\end{eqnarray*}
and counitality:
\begin{eqnarray*}
(\nu_x\ot L_\alpha)\Delta_{1_x,\alpha}&=&(\eta_{1_x,1_x} \ot L_\alpha )(j_x\ot L_{1_x}\ot L_\alpha)(L_{1_x} \ot m_{1_x,\alpha})(\coev_{1_x} \ot L_{\alpha})\\
&=&m_{1_x,\alpha}(\eta_{1_x,1_x}\ot L_{1_x} \ot L_\alpha )(L_{1_x} \ot \coev_{1_x} \ot L_{\alpha})(j_x \ot L_\alpha) \\
&=&m_{1_x,\alpha} (j_x \ot L_\alpha) =L_\alpha,\\
\end{eqnarray*}
and analogously for the right counitality. $(L,\Delta,\nu)$ therefore forms a $(\Gg,\Vv)$-opcategory.\\
We finish by checking the Frobenius conditions:
\begin{eqnarray*}
(L_\alpha \ot m_{\beta,\gamma})(\Delta_{\alpha,\beta}\ot L_ \gamma)&=&(L_\alpha \ot m_{\beta,\gamma})(L_\alpha\ot m_{\alpha^{-1},\alpha\beta}\ot L_\gamma)(\coev_{\alpha^{-1}} \ot L_{\alpha\beta}\ot L_ \gamma)\\
&=&(L_\alpha \ot m_{\alpha^{-1},\alpha\beta\gamma})(L_\alpha\ot L_{\alpha^{-1}}\ot m_{\alpha\beta,\gamma})(\coev_{\alpha^{-1}} \ot L_{\alpha\beta}\ot L_ \gamma)\\
&=&(L_\alpha \ot m_{\alpha^{-1},\alpha\beta\gamma})(\coev_{\alpha^{-1}} \ot m_{\alpha\beta,\gamma})=\Delta_{\alpha,\beta\gamma}m_{\alpha\beta,\gamma}
\end{eqnarray*}
\begin{eqnarray*}
(m_{\alpha,\beta} \ot L_\gamma)(L_\alpha\ot \Delta_{\beta,\gamma})&=&(m_{\alpha,\beta} \ot L_\gamma)(L_\alpha\ot m_{\beta\gamma,\gamma^{-1}} \ot L_\gamma)(L_\alpha \ot L_{\beta\gamma}\ot \coev_{\gamma})\\
&=&(m_{\alpha\beta\gamma,\gamma^{-1}} \ot L_\gamma)(m_{\alpha,\beta\gamma} \ot L_{\gamma^{-1}} \ot L_\gamma)(L_\alpha \ot L_{\beta\gamma}\ot \coev_{\gamma})\\
&=&(m_{\alpha\beta\gamma,\gamma^{-1}} \ot L_\gamma)(m_{\alpha,\beta\gamma} \ot \coev_{\gamma})=\Delta_{\alpha\beta,\gamma}m_{\alpha,\beta\gamma},
\end{eqnarray*}
where we used associativity. 
Therefore $L$ forms a Frobenius $(\Gg,\Vv)-$category in the sense of of \ref{Frob1}. Let $\alpha:x\to y$, then
\begin{equation*}
\nu_x m_{\alpha,\alpha^{-1}}=\eta_{1_x} (j_x \ot m_{\alpha,\alpha^{-1}} )=\eta_{\alpha} (m_{1_x,\alpha}\ot  L_{\alpha^{-1}})(j_x \ot L_{\alpha} \ot L_{\alpha^{-1}})=\eta_{\alpha},
\end{equation*}
by compatibility of multiplication and inner product. Hence $L$ is symmetric if and only if $\eta$ is. 
\end{proof}

We are finally able to define crossed loop Frobenius $(\Gg,\Vv)$-categories and therefore introduce the notion of crossings. By reasons explained in section \ref{2D1} these crossings are defined for all paths, but only act on loops $\alpha: x\to x$ in $\Gg$. This will be the main definition in classifying 2-dimensional $(X,Y)$-HQFTs.

\begin{Definition} \label{trace}
Given a Frobenius $(\Gg,\Vv)$-category $L$ and a morphism $f:L_\alpha \ot L_\beta \to L_\beta$ in $\Vv$ for any $\alpha,\beta \in \Gg$. Its \textit{\textbf{partial trace}} is defined as 
$$\Trace_\beta(f):=\eta_{\beta} (f\ot L_{\beta^{-1}})(L_\alpha \ot \coev_{\beta^{-1}}): L_\alpha \to I.$$
\[
\begin{tikzpicture}
\node (1) at (0,0) {$L_\alpha$};
\node[morphism] (f) at (0,-1) {$\Trace_\beta(f)$};
\draw (1) to [out=-90,in=90] (f.north);
\node (1) at (2,-1) {$=$};
\node (5) at (1,-2) {};
\end{tikzpicture} 
\begin{tikzpicture}
\node (1) at (0.2,0) {$L_\alpha$};
\node (0) at (1.5,0) {};
\node (5) at (1,-2) {};
\node[morphism] (f) at (0.5,-1) {$f$};
\draw (1) to [out=-90,in=90] (f.north west);
\draw (f.south) to [out=-90,in=180] (5.center);
\draw (0.center) to [out=0, in=0] (5.center);
\draw (0.center) to [out=180, in=90] (f.north east);
\end{tikzpicture} 
\]
\end{Definition}

\begin{Definition} \label{loopFrob}
Given a groupoid $\Gg$, consider the totally disconnected groupoid of loops $\Gg_0$. It has the same objects, but only endomorphisms $\alpha:x\to x$. A \textit{\textbf{loop Frobenius $(\Gg,\Vv)$-category}} is a Frobenius $(\Gg_0,\Vv)$-category.
A loop Frobenius $(\Gg,\Vv)$-category is called \textit{\textbf{crossed}}, if there is a family of morphisms called the \textit{\textbf{crossing}} $\phi^\alpha_\beta: L_\alpha\to L_{\beta\alpha\beta^{-1}}$ for all $\alpha:x \to x \in \Gg_0$ and $\beta: y\to x \in \Gg$ such that:
\begin{itemize}
\item[(LF1)] They respect the $(\Gg_0,\Vv)$-category and $(\Gg_0,\Vv)$-opcategory structure:
\begin{eqnarray*}
\phi^{\alpha\alpha'}_\beta m_{\alpha,\alpha'} = m_{\beta \alpha \beta^{-1},\beta \alpha' \beta^{-1}} (\phi^\alpha_\beta \ot \phi^{\alpha'}_\beta), & \phi^{1_x}_\beta j_x=j_y,\\
(\phi^\alpha_\beta \ot \phi^{\alpha'}_\beta) \Delta_{\alpha\alpha'} = \Delta_{\beta\alpha\beta^{-1},\beta\alpha'\beta^{-1}} \phi^{\alpha\alpha'}_\beta, & \nu_y \phi^{1_x}_\beta =\nu_x,
\end{eqnarray*}
for all $\alpha,\alpha':x\to x$ and $\beta:y\to x$.
\item[(LF2)] They are compatible composition:
\begin{eqnarray*}
&\phi^{\beta\alpha\beta^{-1}}_{\gamma} \phi^{\alpha}_{\beta}= \phi^\alpha_{\gamma\beta}, & \phi_{1_x}^\alpha=L_\alpha,
\end{eqnarray*}
for all $\alpha:x\to x$ and $\beta:y\to x, \gamma:z\to y$.
\item[(LF3)] The \textit{\textbf{crossed conditions}} are satisfied:
\begin{eqnarray*}
m_{\beta,\alpha}\sigma=m_{\beta\alpha\beta^{-1},\beta
}(\phi_{\beta}^\alpha\ot L_\beta)&:& L_\alpha\ot L_\beta \to L_{\beta\alpha} \qquad \forall \alpha: x\to x, \beta: x\to x ,\\
\phi_{\alpha}^\alpha=L_\alpha&:& L_\alpha \to L_\alpha \qquad \forall \alpha:x\to x,\\
\Trace_\alpha(m_{\alpha\beta(\beta\alpha)^{-1},\alpha}(L_{\alpha\beta\alpha^{-1}\beta^{-1}} \ot \phi_{\beta}^\alpha))&=&\Trace_\beta(\phi_{\alpha^{-1}}^{\alpha\beta\alpha^{-1}}m_{\alpha\beta(\beta\alpha)^{-1},\beta}( L_{\alpha\beta\alpha^{-1}\beta^{-1}} \ot L_\beta)),
\end{eqnarray*}
for all $\alpha, \beta: x\to x$.
\end{itemize}
A \textit{\textbf{crossed loop Frobenius $(\Gg,\Vv)$-functor}} $F:L\to L'$ is a Frobenius $(\Gg_0,\Vv)$-functor, such that $F(\phi_\beta^\alpha)=\phi'^\alpha_\beta$. Denote the category of loop Frobenius $(\Gg,\Vv)$-categories and -functors by $\sf Frob_{\circ}^{\Gg}(\Vv)$.
\end{Definition}

From this definition general properties of crossed loop Frobenius $(\Gg,\Vv)$-categories can already be deduced. 

\begin{corollary}
For $\alpha:x\to x$ the operation $m_{\alpha,\alpha}: L_\alpha\ot L_\alpha\to L_{\alpha\alpha}$  is commutative. Particularly, $L_\alpha$ is a commutative monoid for any idempotent $\alpha=\alpha^2.$ \\
For $\alpha=1_x$ and $G:=\Gg(x,x)$ there is a left action of $G$ on $L_{1_x}$ by $\{\phi^{1_x}_{\beta}\}_{\beta \in G}$
\end{corollary}

\begin{proposition}
Any crossed loop Frobenius $(\Gg,\Vv)$-category is symmetric as a Frobenius $(\Gg_0,\Vv)$-category.
\end{proposition}
\begin{proof}
First of all note that $\phi_\alpha^\alpha=L_\alpha=\phi_{1_x}^\alpha$ and $\phi_{\alpha^{-1}}^\alpha:L_\alpha\to L_\alpha$ satisfies $\phi_{\alpha^{-1}}^\alpha \phi_\alpha^\alpha= \phi_{1_x}^\alpha$, hence it is the identity on $L_\alpha$. Now
$$m_{\alpha^{-1},\alpha}\sigma=m_{\alpha,\alpha^{-1}} (\phi_{\alpha^{-1}}^\alpha \ot L_\alpha)=m_{\alpha,\alpha^{-1}}.$$
Particularly $\nu_x m_{\alpha^{-1},\alpha} \sigma =\nu_x m_{\alpha,\alpha^{-1}}$.
\end{proof}

\section{Classification of $2$-dimensional $(X,Y)$-HQFTs}\label{2D1}
By the same reasoning as for the classification of $1$-dimensional $(X,Y)$-HQFTs assume some connectedness condition on the pair $(X,Y)$. Fix therefore $(X,Y)$ an admissible pair and $\Vv$ a strict symmetric monoidal target category. We classify HQFTs with target space pair $(X,Y)$ by means of crossed loop Frobenius $(\Gg,\Vv)$-categories. This is done in two steps: determining a crossed loop Frobenius $(\Gg,\Vv)$-category for all 2-dimensional $(X,Y)$-HQFTs, called the \textbf{\textit{underlying}} crossed loop Frobenius $(\Gg,\Vv)$-category. Secondly, showing that every crossed loop Frobenius $(\Gg,\Vv)$-category induces a 2-dimensional $(X,Y)$-HQFT.\\
By Proposition \ref{funcateq2} it is enough to consider a skeleton of the category $\Cob_2^{X,Y}$. Since any connected $1$-dimensional $(X,Y)$-manifold is $(X,Y)$-homeomorphic to the 1-sphere $\Ss^1$ together with a homotopy classes of pointed maps to $X$, identify them with elements $\alpha\in \Gg(x,x)$. Hence, a $2$-dimensional $(X,Y)$-HQFT $\Zz$ maps a connected $(X,Y)$-manifold to an object $L_\alpha$ in $\Vv$. By Proposition \ref{rig} and the fact, that the symmetric monoidal functor $\Zz$ preserves duals, identify $L_{\alpha^{-1}}=L_{\alpha}^*$. Further, for $M$ an arbitrary object in $\Cob^{X,Y}_2$ with the components $\{\alpha_p\}$ its image under $\Zz$ is $\Zz(M)=\bigotimes L_{\alpha_p}$.  \\
Morphisms in $\Cob^{X,Y}_2$ are 2-dimensional $(X,Y)$-cobordisms, called \textbf{\textit{$(X,Y)$-surfaces}}. $\Zz$ maps such an $(X,Y)$-surface $(W,M,N,f)$ with $M=\coprod \alpha_p$ and $N=\coprod \beta_q$ to a morphism $\Zz(W): \bigotimes L_{\alpha_p}\to \bigotimes L_{\beta_q}$ in $\Vv$. Given any $(X,Y)$-surface $W$ label their boundary components with orientations $\epsilon_i \in \{ +,-\}$, depending if the orientation on the boundary agrees with $W$ or not. Hence, the negative compontents will be inputs of $\Zz(W)$ and the positive its output. \\
The easiest $(X,Y)$-surface is the closed disc $B_{\epsilon}(\alpha)$, with orientation on the boundary $\epsilon$ and $g|_{\partial B}=\alpha$ with $\alpha\in\Gg(x,x)$. Note, that given a map $B_\epsilon(\alpha)\to X$ there is a contraction of the loop to the base point, hence $\alpha$ is trivial and there are only two types of closed discs: $B_+(1_x)$ and $B_-(1_x)$.
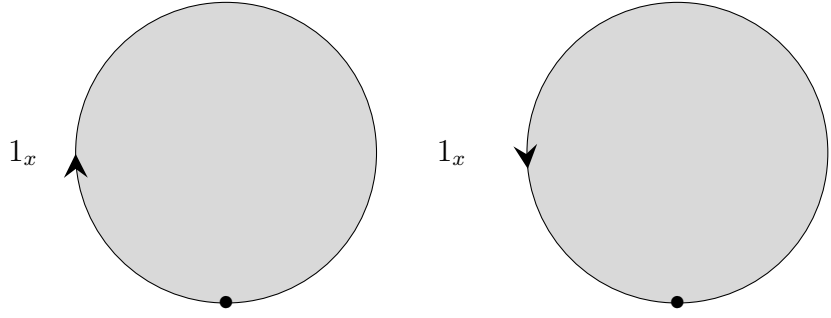
\[
\begin{tikzpicture}
\draw[fill=gray!30,even odd rule]  (0,0) circle (2cm);
\draw[decorate,decoration={markings,mark=at position 6.3 cm with
{\arrowreversed[line width=1mm]{stealth}};}] (0,0) circle (2cm) node at (-2.7,0) {$1_x$};
\draw (0,-2) node{$\bullet$};

\draw[fill=gray!30,even odd rule]  (6,0) circle (2cm) ;
\draw[decorate,decoration={markings,mark=at position 6.5 cm with
{\arrow[line width=1mm]{stealth}};}] (6,0) circle (2cm) node at (3,0) {$1_x$};
\draw (6,-2) node{$\bullet$};
\end{tikzpicture}
\]

Another two elementary $(X,Y)$-surfaces are the annulus and the disc with two holes. Let $C:=\mathbb{S}^1\times [0,1], C^0:=\mathbb{S}^1\times \{0\}$ and $C^1:=\mathbb{S}^1\times \{1\}$ and provide them with base points $(s,0)$ and $(s,1)$ for $s\in \mathbb{S}^1$. Given two signs $\epsilon,\mu=\pm$ we define $C_{\epsilon,\mu}$ the oriented annulus $C$ such that 
$$\partial C_{\epsilon,\mu}=\epsilon C^0_\epsilon\coprod \mu C^1_\mu.$$
The homotopy class of $g: C_{\epsilon,\mu}\to X$ is determined by a loop $\alpha=[g|_{C^0_{\epsilon}}]$ and a path $\beta^{-1}=[g|_{s\times[0,1]}]$. Note that $[g|_{C^1_\mu}]=\beta\alpha^{-\epsilon\mu}\beta^{-1} \in \Gg$. Denote the annulus by $C_{\epsilon,\mu}(\alpha;\beta)$.
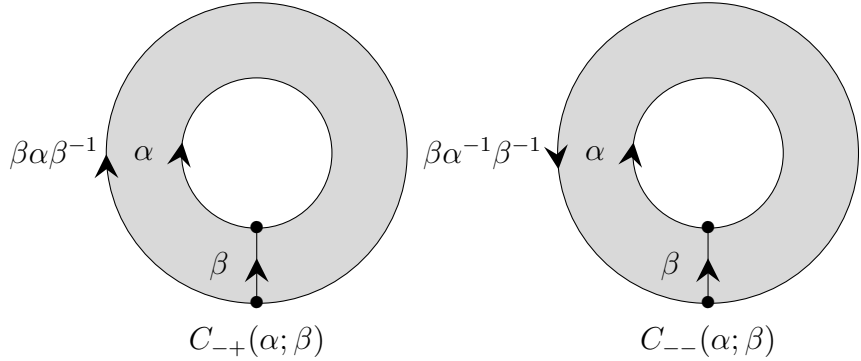
\[
\begin{tikzpicture}
\draw[fill=gray!30,even odd rule]  (0,0) circle (2cm) 
                                (0,0) circle (1cm);
\draw[decorate,decoration={markings,mark=at position 6.3 cm with
{\arrowreversed[line width=1mm]{stealth}};}] (0,0) circle (2cm) node at (-2.7,0) {$\beta\alpha\beta^{-1}$};
\draw[decorate,decoration={ markings,mark=at position 3cm with
{\arrowreversed[line width=1mm]{stealth}};}] (0,0) circle (1cm) node at (-1.5,0) {$\alpha$};
\draw (0,-2) node{$\bullet$};
\draw (0,-1) node{$\bullet$};
\draw[decorate,decoration={ markings,mark=at position 0.6cm with
{\arrow[line width=1mm]{stealth}};}] (0,-2) --(0,-1)  ;
\draw (0,-2) --(0,-1) node at (-0.5,-1.5) {$\beta$} ;
\draw (0,-2.5) node {$C_{-+}(\alpha;\beta)$};

\draw[fill=gray!30,even odd rule]  (6,0) circle (2cm) 
                                (6,0) circle (1cm);
\draw[decorate,decoration={markings,mark=at position 6.5 cm with
{\arrow[line width=1mm]{stealth}};}] (6,0) circle (2cm) node at (3,0) {$\beta\alpha^{-1}\beta^{-1}$};
\draw[decorate,decoration={ markings,mark=at position 3cm with
{\arrowreversed[line width=1mm]{stealth}};}] (6,0) circle (1cm) node at (4.5,0) {$\alpha$};
\draw (6,-2) node{$\bullet$};
\draw (6,-1) node{$\bullet$};
\draw[decorate,decoration={ markings,mark=at position 0.6cm with
{\arrow[line width=1mm]{stealth}};}] (6,-2) --(6,-1)  ;
\draw (6,-2) --(6,-1) node at (5.5,-1.5) {$\beta$} ;
\draw (6,-2.5)  node {$C_{--}(\alpha;\beta)$};
\end{tikzpicture}
\]

Let now $D$ be a 2-disc with two holes, with boundary components $S,T,U$ and base points $s,t,u$. Fix two arcs $us, ut$ in $D$.  Given three signs $\epsilon,\mu,\nu$ we define analgously $D_{\epsilon, \mu,\nu}$ such that 
$$\partial D_{\epsilon,\mu,\nu}=\epsilon S_\epsilon\coprod \mu T_\mu \coprod \nu U_\nu.$$
The homotopy class of $g:D_{\epsilon,\mu,\nu}\to X$ is given by loops $\alpha:=[g|_{S_\epsilon}]$ and $\beta:=[g|_{T_\mu}]$ and two paths $\rho:=[g|_{us}]$ and $ \delta:=[g|_{ut}]$. Denote this $(X,Y)$-surface by $D_{\epsilon,\mu,\nu} (\alpha,\beta;\rho,\delta)$. Note that the loop $\gamma:=[g|_{U_\nu}]$ is given by $(\rho\alpha^{-\epsilon}\rho^{-1}\delta\beta^{-\mu}\delta^{-1})^{\nu}$. There are $(X,Y)$-homeomorphisms
$$D_{\epsilon,\mu,\nu} (\alpha,\beta;\rho,\delta)\simeq D_{\mu,\nu,\epsilon} (\beta,\gamma;\rho^{-1}\delta, \rho^{-1})\simeq D_{\nu,\epsilon,\mu,} (\gamma,\alpha; \delta^{-1},\delta^{-1}\rho).
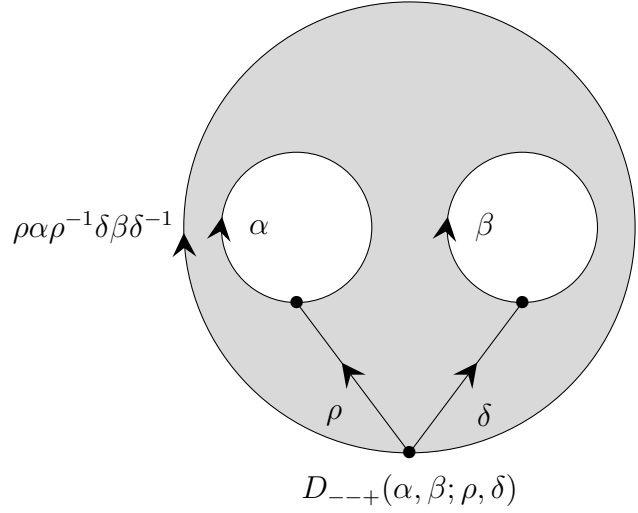
$$

\[
\begin{tikzpicture}
\draw[fill=gray!30,even odd rule]  (0,0) circle (3cm) 
                                (-1.5,0) circle (1cm)
 					(1.5,0) circle (1cm);
\draw[decorate,decoration={markings,mark=at position 9.5 cm with
{\arrowreversed[line width=1mm]{stealth}};}] (0,0) circle (3cm) node at (-4.2,0){$\rho\alpha\rho^{-1}\delta\beta\delta^{-1}$};
\draw[decorate,decoration={ markings,mark=at position 3cm with
{\arrowreversed[line width=1mm]{stealth}};}] (-1.5,0) circle (1cm) node at (-2,0) {$\alpha$};
\draw[decorate,decoration={ markings,mark=at position 3cm with
{\arrowreversed[line width=1mm]{stealth}};}] (1.5,0) circle (1cm) node at (1,0) {$\beta$};
\draw (0,-3) node{$\bullet$};
\draw (-1.5,-1) node{$\bullet$};
\draw (1.5,-1) node{$\bullet$};
\draw[decorate,decoration={ markings,mark=at position 1cm with
{\arrowreversed[line width=1mm]{stealth}};}] (-1.5,-1) --(0,-3)  ;
\draw[decorate,decoration={ markings,mark=at position 1cm with
{\arrowreversed[line width=1mm]{stealth}};}] (1.5,-1) --(0,-3)  ;
\draw (0,-3) --(-1.5,-1) node at (-1,-2.5) {$\rho$} ;
\draw (0,-3) --(1.5,-1) node at (1,-2.5) {$\delta$} ;
\draw (0,-3.5) node {$D_{--+}(\alpha,\beta;\rho,\delta)$};
\end{tikzpicture}
\]

Analogously, denote $(X,Y)$-discs with $n$-holes by $D_{\epsilon_1,...,\epsilon_n,\mu}(\alpha_1,....,\alpha_n; \rho_1,...,\rho_n)$ where $\epsilon_i=\pm$ are the orientations on the inner boundary components and $\mu=\pm$ the orientation on the outer circle. The loops determined by the inner boundary are denote by $\alpha_i$ and the paths connecting the basepoint of the outer circle to the base point of the inner circles are denote by $\rho_i$. \\
The objects $(L_{\alpha})_{\alpha \in \Gg}$ can be endowed with a crossed loop Frobenius $(\Gg,\Vv)$-category structure, using these elementary $(X,Y)$-surfaces. The two closed discs are considered cobordisms from or to the empty manifold, hence the $(X,Y)$-HQFT $\Zz$ maps them to
$$\Zz(B_+(1_x))=:j_x: I\to L_{1_x}, \qquad \Zz(B_-(1_x))=:\nu_x: L_{1_x}\to I$$
Let $\alpha,\beta:x\to x$ be two loops, then $D_{--+} (\alpha,\beta;1_x,1_x)$ is mapped to the morphism 
$$\Zz(D_{--+} (\alpha,\beta;1_x,1_x))=:m_{\alpha,\beta}: L_\alpha\ot L_\beta\to L_{\alpha\beta}, $$
and similarly 
$$\Zz(D_{++-} (\beta,\alpha;1_x,1_x))=:\Delta_{\alpha,\beta}:  L_{\alpha\beta}\to L_{\alpha} \ot L_\beta.$$  
For $\alpha\in \Gg(x,x)$ and $\beta:y\to x$, we can consider the $(X,Y)$-annulus $C_{-+}(\alpha;\beta)$ as a cobordism between $(C^0_{-},\alpha)$ and $(C_+^1,\beta\alpha\beta^{-1})$. Therefore
$$\Zz(C_{-+}(\alpha;\beta))=:\phi^\alpha_\beta: L_\alpha\to L_{\beta\alpha\beta^{-1}}.$$
Particularly, $C_{-+}(\alpha;1_x)$ is $(X,Y)$-homeomorphic to the identity cylinder,  hence
$$\Zz(C_{-+}(\alpha;1_x))=L_\alpha.$$
Furthermore, we define
\begin{eqnarray*}
\Zz(C_{--}(\alpha;1_x))&=:&\eta_\alpha: L_\alpha\ot L_{\alpha^{-1}}\to I.\\
\Zz(C_{++}(\alpha;1_x))&=:&\coev_\alpha: I\to L_{\alpha^{-1} }\ot L_\alpha\to I
\end{eqnarray*}
and we will see in the proof of Lemma \ref{CLF}, that this is consitent with both definitions of Frobenius $(\Gg,\Vv)$-categories. These operations form a crossed loop Frobenius $(\Gg,\Vv)$-category and determine the functor $\Zz$ uniquely.

\begin{theorem}
Let $(\Vv,\ot,I,\sigma)$ be a (strict) symmetric monoidal category. Let $(X,Y)$ be an admissble pair of spaces and denote the relative fundamental groupoid of the pair $(X,Y)$ by $\Gg=\Pi_1(X,Y)$. The category $\Qq_2(X,Y,\Vv)$ of 2-dimensional $(X,Y)$-HQFTs 
 is equivalent to the category $\Frob_{\circ}^{\Gg}(\Vv)$ of crossed loop Frobenius $(\Gg,\Vv)$-categories.
\end{theorem}

The proof is done in multiple lemmas by adapting the proof of Turaev in \cite{Tu1}. 

\begin{lemma}
The objects $(L_\alpha)_{\alpha \in \Gg_0}$ form a $(\Gg_0,\Vv)$-category with multiplication $m_{\alpha,\beta}$ and unit $j_x$.
\end{lemma}

\begin{proof}
We prove associativity and unitality. Let $\alpha,\beta,\gamma \in \Gg(x,x)$ and $$W=D_{---+}(\alpha,\beta,\gamma;1_x,1_x,1_x)$$ be a disc with 3 holes, labeled by $\alpha, \beta, \gamma$ and the three paths connecting the basepoints labeled  by identities.

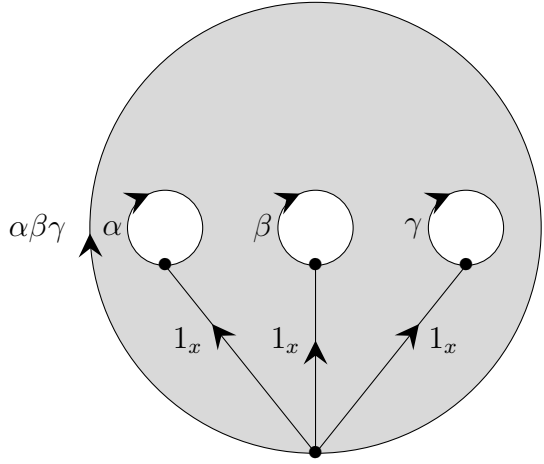
\[
\begin{tikzpicture}
\draw[fill=gray!30,even odd rule]  (0,0) circle (3cm) 
                                (-2,0) circle (0.5cm)
 					(2,0) circle (0.5cm)
					(0,0) circle (0.5cm);
\draw[decorate,decoration={markings,mark=at position 9.5 cm with
{\arrowreversed[line width=1mm]{stealth}};}] (0,0) circle (3cm) node at (-3.7,0){$\alpha\beta\gamma$};
\draw[decorate,decoration={ markings,mark=at position 1cm with
{\arrowreversed[line width=1mm]{stealth}};}] (-2,0) circle (0.5cm) node at (-2.7,0) {$\alpha$};
\draw[decorate,decoration={ markings,mark=at position 1cm with
{\arrowreversed[line width=1mm]{stealth}};}] (0,0) circle (0.5cm) node at (-0.7,0) {$\beta$};
\draw[decorate,decoration={ markings,mark=at position 1cm with
{\arrowreversed[line width=1mm]{stealth}};}] (2,0) circle (0.5cm) node at (1.3,0) {$\gamma$};
\draw (0,-3) node{$\bullet$};
\draw (-2,-0.5) node{$\bullet$};
\draw (0,-0.5) node{$\bullet$};
\draw (2,-0.5) node{$\bullet$};
\draw[decorate,decoration={ markings,mark=at position 1cm with
{\arrowreversed[line width=1mm]{stealth}};}] (-2,-0.5) --(0,-3)  ;
\draw[decorate,decoration={ markings,mark=at position 1cm with
{\arrowreversed[line width=1mm]{stealth}};}] (0,-0.5) --(0,-3)  ;
\draw[decorate,decoration={ markings,mark=at position 1cm with
{\arrowreversed[line width=1mm]{stealth}};}] (2,-0.5) --(0,-3)  ;
\draw (0,-3) --(-2,-0.5) node at (-1.7,-1.5) {$1_x$} ;
\draw (0,-3) --(2,-0.5) node at (1.7,-1.5) {$1_x$} ;
\draw (0,-3) --(0,-0.5) node at (-0.4,-1.5) {$1_x$} ;
\end{tikzpicture}
\]

$W$ can be disected into elementary $(X,Y)$-surfaces in two different ways. Let $$W_0:=D_{--+}(\alpha,\beta;1_x,1_x) \coprod C_{-+}(\gamma;1_x)$$ and $W_1:=D_{--+}(\alpha\beta,\gamma;1_x,1_x)$ and further $$W_2:=C_{-+}(\alpha;1_x)\coprod D_{--+}(\beta,\gamma;1_x,1_x) $$ and $W_3:=D_{--+}(\alpha,\beta\gamma;1_x,1_x)$. Gluing $W_0$ to $W_1$ and $W_2$ to $W_3$ yields discs with 3 holes $(X,Y)$-homeomorphic to $W$. Since $\Zz$ is functorial:
\begin{eqnarray*}
m_{\alpha\beta,\gamma}(m_{\alpha,\beta}\ot L_\gamma)&=&\Zz(D_{--+}(\alpha\beta,\gamma;1_x,1_x))(\Zz(D_{--+}(\alpha,\beta;1_x,1_x)\ot \Zz(C_{-+}(\gamma;1_x)))\\
&=&\Zz(W_1)\Zz(W_0)=\Zz(W)=\Zz(W_3)\Zz(W_2)\\
&=&\Zz(D_{--+}(\alpha,\beta\gamma;1_x,1_x))(\Zz(C_{-+}(\alpha;1_x))\ot \Zz(D_{--+}(\beta,\gamma;1_x,1_x)))\\
&=&m_{\alpha,\beta\gamma}(L_\alpha \ot m_{\beta,\gamma}).
\end{eqnarray*}

For unitality, look at $C_{-+}(\alpha;1_x)$ and the gluing of $D_{--+}(\alpha,1_x;1_x,1_x)$ with $C_{-+}(\alpha;1_x)$ and $B_{+}(1_x)$ glued to the first and second inner boundary respectively. These two $(X,Y)$-surfaces are $(X,Y)$-homeomorphic; hence
$$m_{\alpha,1_x}(L_\alpha \ot j_x)=\Zz(D_{--+}(\alpha,1_x;1_x,1_x))(\Zz(C_{-+}(\alpha;1_x))\ot\Zz(B_{+}(1_x)))=\Zz(C_{-+}(\alpha;1_x))=L_\alpha$$
and analogously for the left unitality. Hence $(L_\alpha)_{\alpha \in \Gg_0}$ form a $(\Gg_0,\Vv)-category$.
\end{proof}

\begin{lemma}
The objects $(L_\alpha)_{\alpha \in \Gg_0}$ form a $(\Gg_0,\Vv)-opcategory$ with comultiplication $\Delta_{\alpha,\beta}$ and counit $\nu_x$.
\end{lemma}
This proof is dual to the previous, by reversing orientations on $D_{--+}$ and $B_+$.

\begin{lemma}
The objects $(L_\alpha)_{\alpha \in \Gg_0}$ form a loop Frobenius $(\Gg,\Vv)$-category.
\end{lemma}
\begin{proof}
We need to prove the Frobenius conditions:
\begin{eqnarray*}
\Delta_{\alpha,\beta\gamma}m_{\alpha\beta,\gamma}&=&(L_\alpha\ot m_{\beta,\gamma})(\Delta_{\alpha,\beta}\ot L_\gamma),\\
\Delta_{\alpha\beta,\gamma}m_{\alpha,\beta\gamma}&=&(m_{\alpha,\beta}\ot L_\gamma)(L_\alpha\ot \Delta_{\beta,\gamma})
\end{eqnarray*}
 for $\alpha,\beta,\gamma \in \Gg(x,x).$
To prove the first condition we look at the $(X,Y)$-disc with three holes $W=D_{--++}(\alpha,\beta\gamma,\gamma;1_x,1_x,1_x)$, which can be represented as gluing $W_1=D_{-++}(\alpha\beta\gamma,\gamma;1_x,1_x)$ $\simeq D_{++-}(\gamma,\alpha\beta;1_x,1_x)$ to $W_2=D_{--+}(\alpha, \beta\gamma;1_x,1_x)$. Similarly, we can represent $W$ as gluing $W_3=D_{--+}(\alpha,\beta;1_x,1_x)\coprod C_{-+}(\gamma,1_x)$ to $W_4=C_{-+}(\alpha;1_x) \coprod D_{-++}(\beta\gamma,\gamma;1_x,1_x)$, where the last disc with two holes is $(X,Y)$-homeomorphic to $D_{++-}(\gamma,\beta;1_x,1_x)$.
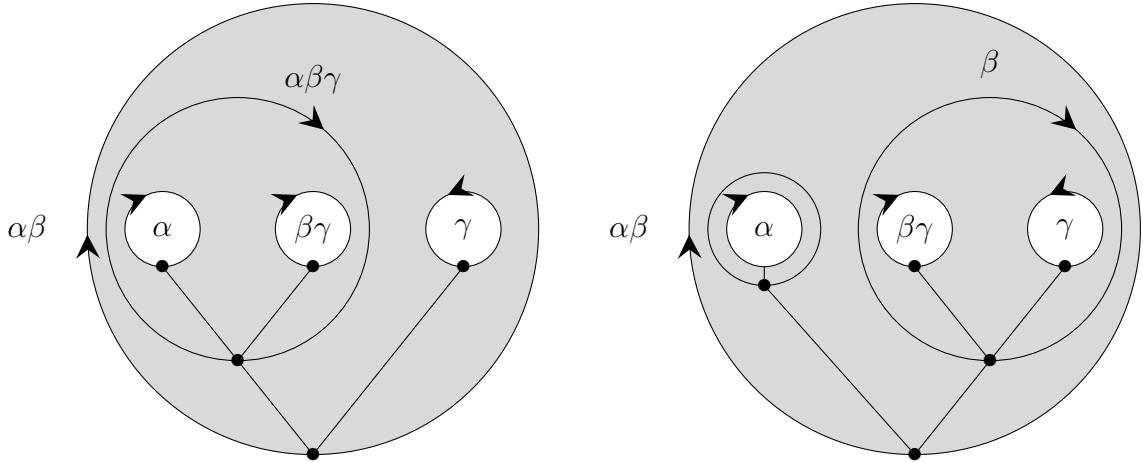
\[
\begin{tikzpicture}
\draw[fill=gray!30,even odd rule]  (0,0) circle (3cm) 
                                (-2,0) circle (0.5cm)
 					(2,0) circle (0.5cm)
					(0,0) circle (0.5cm);
\draw (-1,0) circle (1.75cm);
\draw[decorate,decoration={markings,mark=at position 9.5 cm with
{\arrowreversed[line width=1mm]{stealth}};}] (0,0) circle (3cm) node at (-3.8,0){$\alpha\beta$};
\draw[decorate,decoration={ markings,mark=at position 1cm with
{\arrowreversed[line width=1mm]{stealth}};}] (-2,0) circle (0.5cm) node at (-2,0) {$\alpha$};
\draw[decorate,decoration={ markings,mark=at position 1cm with
{\arrowreversed[line width=1mm]{stealth}};}] (0,0) circle (0.5cm) node at (0,0) {$\beta\gamma$};
\draw[decorate,decoration={ markings,mark=at position 1cm with
{\arrow[line width=1mm]{stealth}};}] (2,0) circle (0.5cm) node at (2,0) {$\gamma$};
\draw[decorate,decoration={ markings,mark=at position 1.5cm with
{\arrowreversed[line width=1mm]{stealth}};}] (-1,0) circle (1.75cm) node at (0,2) {$\alpha\beta\gamma$};
\draw (0,-3) node{$\bullet$};
\draw (-2,-0.5) node{$\bullet$};
\draw (0,-0.5) node{$\bullet$};
\draw (2,-0.5) node{$\bullet$};
\draw (-1,-1.75) node {$\bullet$};
\draw (0,-3) --(-1,-1.75);
\draw (0,-3) --(2,-0.5);
\draw (-1,-1.75) --(0,-0.5) ;
\draw (-1,-1.75) --(-2,-0.5) ;

\draw[fill=gray!30,even odd rule]  (8,0) circle (3cm) 
                                (6,0) circle (0.5cm)
 					(10,0) circle (0.5cm)
					(8,0) circle (0.5cm);
\draw (9,0) circle (1.75cm)
	(6,0) circle (0.75cm);
\draw[decorate,decoration={markings,mark=at position 9.5 cm with
{\arrowreversed[line width=1mm]{stealth}};}] (8,0) circle (3cm) node at (4.2,0){$\alpha\beta$};
\draw[decorate,decoration={ markings,mark=at position 1cm with
{\arrowreversed[line width=1mm]{stealth}};}] (6,0) circle (0.5cm) node at (6,0) {$\alpha$};
\draw[decorate,decoration={ markings,mark=at position 1cm with
{\arrowreversed[line width=1mm]{stealth}};}] (8,0) circle (0.5cm) node at (8,0) {$\beta\gamma$};
\draw[decorate,decoration={ markings,mark=at position 1cm with
{\arrow[line width=1mm]{stealth}};}] (10,0) circle (0.5cm) node at (10,0) {$\gamma$};
\draw[decorate,decoration={ markings,mark=at position 1.5cm with
{\arrowreversed[line width=1mm]{stealth}};}] (9,0) circle (1.75cm) node at (9,2.2) {$\beta$};
\draw (8,-3) node{$\bullet$};
\draw (6,-0.75) node{$\bullet$};
\draw (8,-0.5) node{$\bullet$};
\draw (10,-0.5) node{$\bullet$};
\draw (9,-1.75) node {$\bullet$};
\draw (8,-3) --(9,-1.75);
\draw (8,-3) --(6,-0.75);
\draw (9,-1.75) --(8,-0.5) ;
\draw (9,-1.75) --(10,-0.5) ;
\draw (6,-0.75)--(6,-0.5);
\end{tikzpicture}
\]

Hence
\begin{eqnarray*}
\Delta_{\alpha\beta,\gamma}m_{\alpha,\beta\gamma}&=&\Zz(D_{+-+}(\alpha\beta\gamma,\gamma;1_x,1_x))\Zz(D_{--+}(\alpha, \beta\gamma;1_x,1_x))=\Zz(W)\\
&=&\Zz\left(D_{--+}(\alpha,\beta;1_x,1_x)  \coprod C_{-+}(\gamma,1_x)\right)\Zz\left(C_{-+}(\alpha;1_x) \coprod D_{-++}(\beta\gamma,\gamma;1_x,1_x)\right)\\
&=&(m_{\alpha,\beta}\ot L_\gamma)(L_\alpha\ot \Delta_{\beta,\gamma}).
\end{eqnarray*}
Similar argumentation on $D_{+--+}(\alpha,\alpha\beta,\gamma;1_x,1_x,1_x)$ shows that $$\Delta_{\alpha,\beta\gamma}m_{\alpha\beta,\gamma}=(L_\alpha\ot m_{\beta,\gamma})(\Delta_{\alpha,\beta}\ot L_\gamma).$$
\end{proof}
\begin{lemma}
Given some 2-punctured $(X,Y)$-disc $D_{--+}(\alpha,\beta;\rho,\delta)$, such that the output label is $\gamma=\rho\alpha\rho^{-1}\delta\beta\rho^{-1}=\beta$, let $P$ be the $(X,Y)$-surface obtained from gluing together the two boundary components labeled $\beta$. Then
$$\Zz(P)=\Trace_{\beta}(\Zz(D_{--+}(\alpha,\beta;\rho,\delta))).$$
\end{lemma}

\begin{proof}
The $(X,Y)$-surface is $(X,Y)$-homeomorphic to gluing $D_{--+}(\alpha,\beta;\rho,\delta)\coprod C_{-+}(\beta^{-1};1_x)$ to $C_{-+}(\alpha;1_x)\coprod C_{--}(\beta;1_x)) $ and $C_{--}(\beta;1_x)$. 
\[
\begin{tikzpicture}[tqft/cobordism/.style={draw}]
\pic[tqft/cap, name=d];
\pic[tqft/pair of pants, name=b, anchor=icoming boundary 1, at=(d-outgoing boundary 1)];
\pic[tqft/pair of pants, fill, name=a, anchor=incoming boundary, at=(b-outgoing boundary 1)];
\pic[tqft/cylinder to next,anchor=incoming boundary, at=(b-outgoing boundary 2)];
\pic[tqft/reverse pair of pants,name=r, anchor=incoming boundary 1, at=(a-outgoing boundary 2)];
\pic[tqft/cup, anchor=incoming boundary 1, at=(r-outgoing boundary 1)];
\end{tikzpicture}
\]
This yields
$$\Zz(P)=\eta_{\beta} (\Zz(D_{--+}(\alpha,\beta;\rho,\delta))\ot L_{\beta^{-1}})(L_\alpha \ot \coev_{\beta^{-1}}),$$
hence it is the partial trace as define in \ref{trace}.
\end{proof}
\begin{lemma}
The loop Frobenius $(\Gg,\Vv)$-category $(L,m,j,\Delta,\nu)$ is crossed via 
$\Zz(C_{-+}(\alpha;\beta))=\phi^\alpha_\beta.$
\end{lemma}
\begin{proof}
We turn our attention to $C_{-+}(\alpha;\beta)$ where $\alpha:x\to x$ is a loop and $\beta:y\to x $ a path in $X$ with $x,y\in Y$. The family $\phi^\alpha_\beta$ satisfies all axioms for a crossing of the loop Frobenius $(\Gg,\Vv)$-category $(L_\alpha)_{\alpha_\in \Gg_0}$.\\
We begin by showing compatibility with the $(\Gg_0,\Vv)$-category and $(\Gg_0,\Vv)$-opcategory structure:
The $X$-disc $B_{+}(1_y)$ is mapped to $j_y:  I\to L_{1_y}$ by $\Zz$. It can also be obtained by gluing $B_{+}(1_x)$ to $C_{-+}(1_x;\beta)$. Similarly $B_{-}(1_y)\simeq B_{-}(1_x)\coprod_{\partial B_{-}(1_x)} C_{+-}(1_x;\beta)$

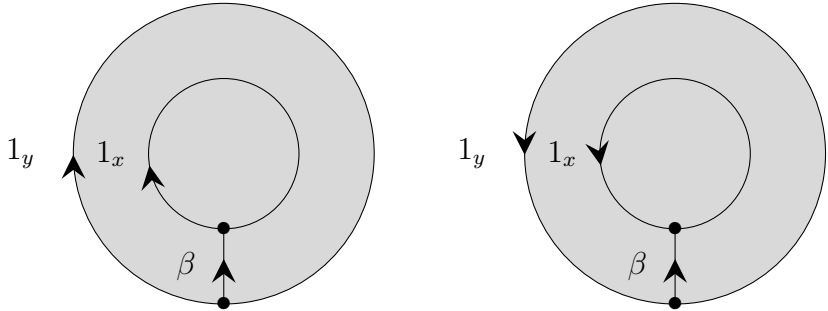
\[
\begin{tikzpicture}
\draw[fill=gray!30,even odd rule]  (0,0) circle (2cm);
\draw (0,0) circle (1cm);
\draw[decorate,decoration={markings,mark=at position 6.3 cm with
{\arrowreversed[line width=1mm]{stealth}};}] (0,0) circle (2cm) node at (-2.7,0) {$1_y$};
\draw[decorate,decoration={markings,mark=at position 3.3 cm with
{\arrowreversed[line width=1mm]{stealth}};}] (0,0) circle (1cm) node at (-1.5,0) {$1_x$};
\draw (0,-2) node {$\bullet$}
(0,-1) node {$\bullet$};
\draw[decorate,decoration={ markings,mark=at position 0.6cm with
{\arrow[line width=1mm]{stealth}};}] (0,-2) --(0,-1)  ;
\draw (0,-2) --(0,-1) node at (-0.5,-1.5) {$\beta$} ;

\draw[fill=gray!30,even odd rule]  (6,0) circle (2cm);
\draw (6,0) circle (1cm);
\draw[decorate,decoration={markings,mark=at position 6.3 cm with
{\arrow[line width=1mm]{stealth}};}] (6,0) circle (2cm) node at (3.3,0) {$1_y$};
\draw[decorate,decoration={markings,mark=at position 3.3 cm with
{\arrow[line width=1mm]{stealth}};}] (6,0) circle (1cm) node at (4.5,0) {$1_x$};
\draw (6,-2) node {$\bullet$}
(6,-1) node {$\bullet$};
\draw[decorate,decoration={ markings,mark=at position 0.6cm with
{\arrow[line width=1mm]{stealth}};}] (6,-2) --(6,-1)  ;
\draw (6,-2) --(6,-1) node at (5.5,-1.5) {$\beta$} ;
\end{tikzpicture}
\]

This shows $\phi^{1_x}_\beta j_x= j_y$ and $\nu_x \phi^{1_y}_{\beta^{-1}}= \nu_y$, where we used that $C_{+-}(1_x;\beta)\simeq C_{-+}(1_y;\beta^{-1}).$ For (co)multiplication consider $D_{--+}(\alpha,\alpha';\beta,\beta)$ and $D_{++-}(\beta\alpha'\beta^{-1},\beta\alpha\beta^{-1};\beta^{-1},\beta^{-1})$ respectively. We can represent $D_{--+}(\alpha,\alpha';\beta,\beta)$ either as the gluing of $D_{--+}(\beta\alpha\beta^{-1},\beta\alpha'\beta^{-1};$ $1_y,1_y)$ and $C_{-+}(\alpha,\beta)\coprod C_{-+}(\alpha',\beta)$ along the inner circles or $D_{--+}(\alpha,\alpha';1_x,1_x)$ and $C_{-+}(\alpha\alpha';$ $\beta)$ along the outer circle, hence $\phi^{\alpha}_\beta m_{\alpha,\alpha'}=m_{\beta \alpha \beta^{-1}, \beta \alpha' \beta^{-1}} (\phi^{\alpha}_\beta \ot \phi^{\alpha'}_\beta).$

\[
\begin{tikzpicture}
\draw[fill=gray!30,even odd rule]  (0,0) circle (3cm) 
                                (-1.5,0) circle (0.5cm)
 					(1.5,0) circle (0.5cm);
\draw (-1.5,0) circle (1cm)
(1.5,0) circle (1cm);
\draw[decorate,decoration={markings,mark=at position 9.5 cm with
{\arrowreversed[line width=1mm]{stealth}};}] (0,0) circle (3cm) node at (-4,0) {$\beta\alpha\alpha'\beta^{-1}$};
\draw[decorate,decoration={ markings,mark=at position 1.5cm with
{\arrowreversed[line width=1mm]{stealth}};}] (-1.5,0) circle (0.5cm) node at (-1.5,0) {$\alpha$};
\draw[decorate,decoration={ markings,mark=at position 1.5cm with
{\arrowreversed[line width=1mm]{stealth}};}] (1.5,0) circle (0.5cm) node at (1.5,0) {$\alpha'$};
\draw[decorate,decoration={ markings,mark=at position 3cm with
{\arrowreversed[line width=1mm]{stealth}};}] (-1.5,0) circle (1cm) node at (-1.5,1.3) {$\beta\alpha\beta^{-1}$};
\draw[decorate,decoration={ markings,mark=at position 3cm with
{\arrowreversed[line width=1mm]{stealth}};}] (1.5,0) circle (1cm) node at (1.5,1.3) {$\beta\alpha'\beta^{-1}$};
\draw (0,-3) node{$\bullet$};
\draw (-1.5,-1) node{$\bullet$};
\draw (1.5,-1) node{$\bullet$};
\draw (-1.5,-0.5) node{$\bullet$};
\draw (1.5,-0.5) node{$\bullet$};
\draw[decorate,decoration={ markings,mark=at position 1cm with
{\arrowreversed[line width=1mm]{stealth}};}] (-1.5,-1) --(0,-3)  ;
\draw[decorate,decoration={ markings,mark=at position 1cm with
{\arrowreversed[line width=1mm]{stealth}};}] (1.5,-1) --(0,-3)  ;
\draw (0,-3) --(-1.5,-1) node at (-1,-2.5) {$1_y$} ;
\draw (0,-3) --(1.5,-1) node at (1,-2.5) {$1_y$} ;
\draw (-1.5,-0.5) --(-1.5,-1) node at (-1.3,-0.75) {$\beta$} ;
\draw (1.5,-0.5) --(1.5,-1) node at (1.7,-0.75) {$\beta$} ;

\draw[fill=gray!30,even odd rule]  (8,0) circle (3cm) 
                                (6.5,0) circle (0.5cm)
 					(9.5,0) circle (0.5cm);
\draw (8,0) circle (2.5cm);
\draw[decorate,decoration={markings,mark=at position 9.5 cm with
{\arrowreversed[line width=1mm]{stealth}};}] (8,0) circle (3cm) node at (4,0) {$\beta\alpha\alpha'\beta^{-1}$};
\draw[decorate,decoration={ markings,mark=at position 1.5cm with
{\arrowreversed[line width=1mm]{stealth}};}] (6.5,0) circle (0.5cm) node at (6.5,0) {$\alpha$};
\draw[decorate,decoration={ markings,mark=at position 1.5cm with
{\arrowreversed[line width=1mm]{stealth}};}] (9.5,0) circle (0.5cm) node at (9.5,0) {$\alpha'$};
\draw[decorate,decoration={ markings,mark=at position 2.5cm with
{\arrowreversed[line width=1mm]{stealth}};}] (8,0) circle (2.5cm) node at (8,2) {$\alpha\alpha'$};
\draw (8,-3) node{$\bullet$};
\draw (8,-2.5) node{$\bullet$};
\draw (6.5,-0.5) node{$\bullet$};
\draw (9.5,-0.5) node{$\bullet$};
\draw[decorate,decoration={ markings,mark=at position 1cm with
{\arrowreversed[line width=1mm]{stealth}};}] (6.5,-0.5) --(8,-2.5)  ;
\draw[decorate,decoration={ markings,mark=at position 1cm with
{\arrowreversed[line width=1mm]{stealth}};}] (9.5,-0.5) --(8,-2.5)  ;
\draw (8,-2.5) --(6.5,-0.5) node at (7,-1.7) {$1_x$} ;
\draw (8,-2.5) --(9.5,-0.5) node at (9.2,-1.7) {$1_x$} ;
\draw (8,-3) --(8,-2.5) node at (8.3,-2.75) {$\beta$} ;
\end{tikzpicture}
\]

Similarly, for $D_{++-}(\beta\alpha'\beta^{-1},\beta\alpha\beta^{-1};\beta^{-1},\beta^{-1})$ we have 
\begin{eqnarray*}
&&\left (C_{+-}(\beta\alpha\beta^{-1},\beta^{-1})\coprod C_{+-}(\beta\alpha'\beta^{-1},\beta^{-1})\right ) \coprod_{\partial C\coprod \partial C} D_{++-}(\alpha',\alpha;1_x,1_x)\\
&=&D_{++-}(\beta\alpha'\beta^{-1},\beta\alpha\beta^{-1},1_y,1_y)\coprod_{\partial C} C_{+-}(\beta\alpha\alpha'\beta^{-1},\beta^{-1}).
\end{eqnarray*}

\[
\begin{tikzpicture}
\draw[fill=gray!30,even odd rule]  (0,0) circle (3cm) 
                                (-1.5,0) circle (0.5cm)
 					(1.5,0) circle (0.5cm);
\draw (-1.5,0) circle (1.3cm)
(1.5,0) circle (1.3cm);
\draw[decorate,decoration={markings,mark=at position 9.5 cm with
{\arrow[line width=1mm]{stealth}};}] (0,0) circle (3cm) node at (-3.7,0.7) {$\alpha\alpha'$};
\draw[decorate,decoration={ markings,mark=at position 1.5cm with
{\arrow[line width=1mm]{stealth}};}] (-1.5,0) circle (0.5cm) node at (-1.3,0.7) {$\beta\alpha'\beta^{-1}$};
\draw[decorate,decoration={ markings,mark=at position 1.5cm with
{\arrow[line width=1mm]{stealth}};}] (1.5,0) circle (0.5cm) node at (1.7,0.7) {$\beta\alpha\beta^{-1}$};
\draw[decorate,decoration={ markings,mark=at position 3cm with
{\arrow[line width=1mm]{stealth}};}] (-1.5,0) circle (1.3cm) node at (-1.5,1.5) {$\alpha'$};
\draw[decorate,decoration={ markings,mark=at position 3cm with
{\arrow[line width=1mm]{stealth}};}] (1.5,0) circle (1.3cm) node at (1.5,1.5) {$\alpha$};
\draw (0,-3) node{$\bullet$};
\draw (-1.5,-1.3) node{$\bullet$};
\draw (1.5,-1.3) node{$\bullet$};
\draw (-1.5,-0.5) node{$\bullet$};
\draw (1.5,-0.5) node{$\bullet$};
\draw[decorate,decoration={ markings,mark=at position 1cm with
{\arrowreversed[line width=1mm]{stealth}};}] (-1.5,-1.3) --(0,-3)  ;
\draw[decorate,decoration={ markings,mark=at position 1cm with
{\arrowreversed[line width=1mm]{stealth}};}] (1.5,-1.3) --(0,-3)  ;
\draw (0,-3) --(-1.5,-1.3) node at (-1,-2.5) {$1_x$} ;
\draw (0,-3) --(1.5,-1.3) node at (1,-2.5) {$1_x$} ;
\draw (-1.5,-0.5) --(-1.5,-1.3) node at (-1.1,-0.75) {$\beta^{-1}$} ;
\draw (1.5,-0.5) --(1.5,-1.3) node at (1.9,-0.75) {$\beta^{-1}$} ;

\draw[fill=gray!30,even odd rule]  (8,0) circle (3cm) 
                                (6.5,0) circle (0.5cm)
 					(9.5,0) circle (0.5cm);
\draw (8,0) circle (2.5cm);
\draw[decorate,decoration={markings,mark=at position 9.5 cm with
{\arrow[line width=1mm]{stealth}};}] (8,0) circle (3cm) node at (3.7,0) {$\alpha\alpha'$};
\draw[decorate,decoration={ markings,mark=at position 1.5cm with
{\arrow[line width=1mm]{stealth}};}] (6.5,0) circle (0.5cm) node at (6.7,0.7) {$\beta\alpha'\beta^{-1}$};
\draw[decorate,decoration={ markings,mark=at position 1.5cm with
{\arrow[line width=1mm]{stealth}};}] (9.5,0) circle (0.5cm) node at (9.7,0.7) {$\beta\alpha\beta^{-1}$};
\draw[decorate,decoration={ markings,mark=at position 2.5cm with
{\arrow[line width=1mm]{stealth}};}] (8,0) circle (2.5cm) node at (8,2) {$\beta\alpha\alpha'\beta^{-1}$};
\draw (8,-3) node{$\bullet$};
\draw (8,-2.5) node{$\bullet$};
\draw (6.5,-0.5) node{$\bullet$};
\draw (9.5,-0.5) node{$\bullet$};
\draw[decorate,decoration={ markings,mark=at position 1cm with
{\arrowreversed[line width=1mm]{stealth}};}] (6.5,-0.5) --(8,-2.5)  ;
\draw[decorate,decoration={ markings,mark=at position 1cm with
{\arrowreversed[line width=1mm]{stealth}};}] (9.5,-0.5) --(8,-2.5)  ;
\draw (8,-2.5) --(6.5,-0.5) node at (7,-1.7) {$1_y$} ;
\draw (8,-2.5) --(9.5,-0.5) node at (9.2,-1.7) {$1_y$} ;
\draw (8,-3) --(8,-2.5) node at (8.4,-2.75) {$\beta^{-1}$} ;
\end{tikzpicture}
\]
 Since $C_{+-}(\beta\alpha\beta^{-1},\beta^{-1})\simeq C_{-+}(\alpha,\beta)$, this results in 
\begin{eqnarray*}
 (\phi^{\alpha}_\beta \ot \phi^{\alpha'}_\beta)  \Delta_{\alpha\alpha'} =\Delta_{\beta\alpha\beta^{-1},\beta\alpha'\beta^{-1}} \phi^{\alpha\alpha'}_\beta.
\end{eqnarray*}
We have $\phi_{1_x}^\alpha=\Zz(C_{-+}(\alpha;1_x))=L_\alpha$ and regarding the composition
\begin{eqnarray*}
\phi^\alpha_{\gamma\beta}&=&\Zz(C_{-+}(\alpha;\gamma\beta))= \Zz(C_{-+}(\beta\alpha\beta^{-1};\gamma))\Zz(C_{-+}(\alpha;\beta))\\
&=&\phi^{\beta\alpha\beta^{-1}}_{\gamma} \phi^{\alpha}_{\beta}.
\end{eqnarray*}
A Dehn twist about the circle $\mathbb{S} \times \{1/2\} \subset C_{-+}(\alpha;1_x)$ defines an $(X,Y)$-homeomorphism $C_{-+}(\alpha;1_x)\simeq C_{-+}(\alpha,\alpha)$ and since $\Zz$ is invariant under $(X,Y)$-homeomorphisms we have $\phi_{\alpha}^\alpha=L_\alpha$ for all $\alpha:x\to x$.\\
Now for $\alpha,\beta:x\to x$  look at $D_{--+}(\alpha,\beta;1_x,1_x)$, but $(X,Y)$-homeomorphically switch the positions of the inner circles:

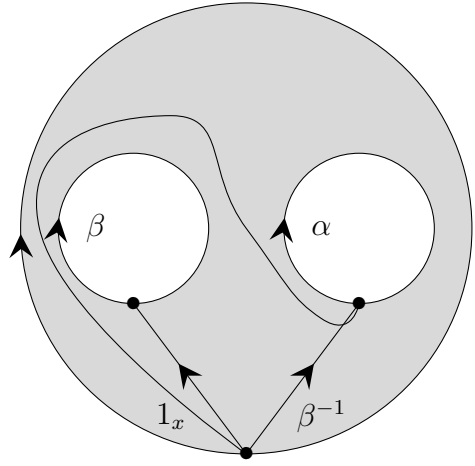
\[
\begin{tikzpicture}
\draw[fill=gray!30,even odd rule]  (0,0) circle (3cm) 
                                (-1.5,0) circle (1cm)
 					(1.5,0) circle (1cm);
\draw[decorate,decoration={markings,mark=at position 9.5 cm with
{\arrowreversed[line width=1mm]{stealth}};}] (0,0) circle (3cm) node at (-4,0){$$};
\draw[decorate,decoration={ markings,mark=at position 3cm with
{\arrowreversed[line width=1mm]{stealth}};}] (-1.5,0) circle (1cm) node at (-2,0) {$\beta$};
\draw[decorate,decoration={ markings,mark=at position 3cm with
{\arrowreversed[line width=1mm]{stealth}};}] (1.5,0) circle (1cm) node at (1,0) {$\alpha$};
\draw (0,-3) node{$\bullet$};
\draw (-1.5,-1) node{$\bullet$};
\draw (1.5,-1) node{$\bullet$};
\draw[decorate,decoration={ markings,mark=at position 1cm with
{\arrowreversed[line width=1mm]{stealth}};}] (-1.5,-1) --(0,-3)  ;
\draw[decorate,decoration={ markings,mark=at position 1cm with
{\arrowreversed[line width=1mm]{stealth}};}] (1.5,-1) --(0,-3)  ;
\draw (0,-3) --(-1.5,-1) node at (-1,-2.5) {$1_x$} ;
\draw (0,-3) --(1.5,-1) node at (1,-2.5) {$\beta^{-1}$} ;
\draw
plot [smooth, tension=1] coordinates {(0,-3) (-2.75,0) (-1,1.5) (0,0) (1,-1.2) (1.5,-1)};
\draw (0,-3.5) node {$D_{--+}(\alpha,\beta;1_x,1_x)\simeq D_{--+}(\beta,\alpha;1_x,\beta^{-1})\sigma$};
\end{tikzpicture}
\]

This is a $(X,Y)$-homeomorphic to $D_{--+}(\beta,\alpha;1_x,\rho)\sigma$ for some $\rho:x\to x$ and since the composition $1_x\beta1_x^{-1}\rho=1_x$, the homotopy class of $\rho$ is the same as of $\beta^{-1}$. Hence 
\begin{eqnarray*}
m_{\alpha,\beta}&=&\Zz(D_{--+}(\alpha,\beta,1_x,1_x))=\Zz(D_{--+}(\beta,\alpha;1_x,\beta^{-1}))\sigma=\\
&=&\Zz(D_{--+}(\beta,\beta^{-1}\alpha\beta,1_x,1_x)) (\Zz(C_{-+}(\beta,1_x))\ot \Zz(C_{-+}(\alpha,\beta)))\sigma\\
&=& m_{\beta,\beta^{-1}\alpha\beta}(L_\beta \ot \phi^\alpha_{\beta^{-1}})\sigma.
\end{eqnarray*}

We apply this to $m_{\beta\alpha\beta^{-1},\beta}$ to get
\begin{eqnarray*} 
m_{\beta\alpha\beta^{-1}, \beta}(\phi_\beta^{\alpha}\ot L_\beta)&=&m_{\beta,\alpha}\sigma (\phi^{\beta\alpha\beta^{-1}}_{\beta^{-1}}\ot L_\beta)(\phi_\beta^{\alpha}\ot L_\beta)\\
&=&m_{\beta,\alpha}\sigma(\phi_{\beta^{-1}\beta}^{\alpha}\ot L_\beta)=m_{\beta,\alpha}\sigma,
\end{eqnarray*}
as in the form of the definition of a crossing on loop Frobenius $(\Gg,\Vv)$-categories.
We are only left to prove the crossed condition:
\begin{eqnarray*}
&\Trace_\alpha(m_{\alpha\beta(\beta\alpha)^{-1},\alpha}(L_{\alpha\beta\alpha^{-1}\beta^{-1}} \ot \phi_{\beta}^\alpha))=&\Trace_\beta(\phi_{\alpha^{-1}}^{\alpha\beta\alpha^{-1}}m_{\alpha\beta(\beta\alpha)^{-1},\beta}( L_{\alpha\beta\alpha^{-1}\beta^{-1}} \ot L_\beta)).
\end{eqnarray*}
Consider the punctured torus $P=\mathbb{S}^1\times \mathbb{S}^1 \setminus B$. Let $(s,s)$ be the basepoint of the circle and deform the hole, such that it does not intersect $\mathbb{S}^1\times \{s\}$ and $\{s\} \times \mathbb{S}^1$. Further let $g:P\to X$ with $g(s,s)=x\in Y$, then $g$ is fully determined by $\alpha,\beta: x\to x$ in $\Gg$ with $\alpha=g|_{\mathbb{S}\times \{s\}}$ and $\beta=g|_{\{s\} \times \mathbb{S}}$ and $\partial P$ is represented by the loop $\alpha\beta\alpha^{-1}\beta^{-1}$. 
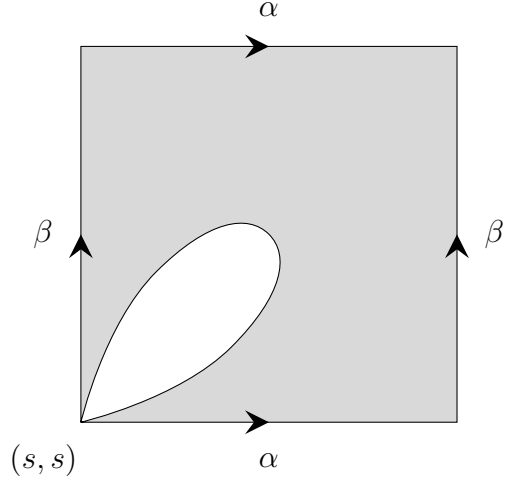
\[
\begin{tikzpicture}
\draw[fill=gray!30, even odd rule]  (0,0) rectangle (5,5)
plot [smooth, tension=1] coordinates {(0,0) (1,2) (2.5,2.5) (2,1) (0,0)};
\draw[decorate,decoration={markings,mark=at position 2.5 cm with
{\arrow[line width=1mm]{stealth}};}] (0,0) --(5,0) node at (2.5,-0.5) {$\alpha$};
\draw[decorate,decoration={ markings,mark=at position 2.5cm with
{\arrow[line width=1mm]{stealth}};}] (0,0) --(0,5) node at (-0.5,2.5) {$\beta$};
\draw[decorate,decoration={markings,mark=at position 2.5 cm with
{\arrow[line width=1mm]{stealth}};}] (5,0) --(5,5) node at (5.5,2.5) {$\beta$};
\draw[decorate,decoration={markings,mark=at position 2.5 cm with
{\arrow[line width=1mm]{stealth}};}] (0,5) --(5,5) node at (2.5,5.5) {$\alpha$};
\draw (-0.5,-0.5) node {$(s,s)$};
\end{tikzpicture}
\]

Hence $\Zz(P): L_{\alpha\beta\alpha^{-1}\beta^{-1}}\to I$. On one hand $P$ can be constructed by taking $D_{--+}(\alpha\beta\alpha^{-1}\beta^{-1},$ $\alpha; 1_x,\beta)$ and gluing the two circles labeled by $\alpha$ together. Since
$$ \Zz(D_{--+}(\alpha\beta\alpha^{-1}\beta^{-1},\alpha; 1_x,\beta))=m_{\alpha\beta\alpha^{-1}\beta^{-1},\beta\alpha\beta^{-1}}(L_{\alpha\beta\alpha^{-1}\beta^{-1}} \ot \phi^\alpha_\beta): L_{\alpha\beta\alpha^{-1}\beta^{-1}}\ot L_\alpha \to L_\alpha, $$
 by the previous lemma
$$\Zz(P)=\Trace_\alpha(m_{\alpha\beta(\beta\alpha)^{-1},\alpha}(L_{\alpha\beta\alpha^{-1}\beta} \ot \phi_{\beta}^\alpha)).$$
On the other hand, $P$ can be obtained by taking $D_{--+}(\alpha\beta\alpha^{-1}\beta^{-1},\beta; \alpha^{-1}, \alpha^{-1})$ and gluing the two circles labeled by $\beta$ together. 
\begin{eqnarray*}
\Zz(D_{--+}(\alpha\beta\alpha^{-1}\beta^{-1},\beta; \alpha^{-1}, \alpha^{-1}))&=&m_{\beta\alpha^{-1}\beta^{-1}\alpha^{-1},\alpha\beta\alpha^{-1}}(\phi_{\alpha^{-1}}^{\alpha\beta\alpha^{-1}\beta^{-1}}\ot \phi_{\alpha^{-1}}^\beta)\\
&=&\phi_{\alpha^{-1}}^{\alpha\beta\alpha^{-1}}m_{\alpha\beta(\beta\alpha)^{-1},\beta} : L_{\alpha\beta\alpha^{-1}\beta^{-1}}\ot L_\beta \to L_\beta, 
\end{eqnarray*}
and therefore 
 $$\Zz(P)=\Trace_\beta (\phi_{\alpha^{-1}}^{\alpha\beta\alpha^{-1}}m_{\alpha\beta(\beta\alpha)^{-1},\beta}).$$
\end{proof}

\begin{lemma} \label{CLF}
The functor $F:\Qq_2(X,Y,\Vv)\to \Frob_{\circ}^{\Gg}(\Vv)$ assigning to every HQFT $\Zz:\Cob_{2}^{X,Y}\to \Vv$ its underlying crossed loop Frobenius $(\Gg,\Vv)$-category is an equivalence of categories. 
\end{lemma}

\begin{proof}
Firstly the assignement is functorial. Let $\rho:\Zz\Rightarrow\Zz'$ be a morphism of HQFT, hence $\rho$ is a monoidal natural isomorphism. Explicitely, $\rho$ is given by a family of maps $\rho_{\alpha}: L_\alpha\to L'_\alpha$, given by $\Zz(\mathbb{S},\alpha)=:L_\alpha$, $\Zz'(\mathbb{S},\alpha)=:L'_\alpha$, satisfying:
\[
\xymatrix{
\bigotimes_{i=1}^n L_{\alpha_i} \ar[rr]^{\bigotimes_{i=1}^m \rho_{\alpha_i}} \ar[d]_{\Zz(W,g)}&  &\bigotimes_{i=1}^n L'_\alpha \ar[d]^{\Zz'(W,g)}\\
\bigotimes_{j=1}^m L_{\beta_j}  \ar[rr]_{\bigotimes_{j=1}^m \rho_{\beta_j}}& & \bigotimes_{j=1}^m L_{\beta_j},
}
\]
for all $(W,g)$ cobordant to $\coprod_{i=1}^n (\mathbb{S},\alpha_i)$  and $\coprod_{j=1}^n (\mathbb{S},\beta_j)$. Particularly, that means $\rho$ is a crossed loop Frobenius functor. Further the functor $F$ is faithful. To show that $F$ is full, we show that any morphism $\Zz(W,g)$ can be represented by a composition of the structural morphisms $(m,j,\Delta,\nu,\phi)$ (and $\eta,\coev$ respectively) in $L$.\\
One shows, that it is essentially surjective, by constructing an HQFT given an arbitrary crossed loop Frobenius $(\Gg,\Vv)$-category. We use the fact that any $(X,Y)$-surface can be obtained by gluing discs with $\leq 2$ holes and by connectivity condition we can assume that such cutting along circles always go through some point sent to $Y$. Define $\Zz$ on $B,C,D$ and show that this assignement is independent of the gluing, hence obtaining a full $(X,Y)$-HQFT. \\
For the disc with no holes we have $\Zz(B_{-}(1_x)):=\nu_x$ and $\Zz(B_{+}(1_x)):=j_x$.\\
$(X,Y)$-discs with one hole are $(X,Y)$-homeomorphic to one of the four $C_{--}(\alpha;\beta)$, $C_{-+}(\alpha;\beta)$, $C_{+-}(\alpha;\beta)\simeq C_{-+}(\beta\alpha\beta^{-1};\beta^{-1})$ or $C_{++}(\alpha;\beta)$. 
For $\alpha\in \Gg(x,x)$ the $(X,Y)$-annulus $C_{--}(\alpha; 1_x)$ is an $(X,Y)$-cobordisms between $(C^0_{-},\alpha)\coprod (C^1_{-},\alpha^{-1})$ and $\emptyset_1$. Therefore 
$$\Zz(C_{--}(\alpha;1_x)):=\eta_\alpha: L_\alpha\ot L_{\alpha^{-1}}\to I$$
and similarly, 
$$\Zz(C_{++}(\alpha;1_x)):=\coev_\alpha: I\to L_{\alpha^{-1}}\ot L_\alpha.$$
Let
$$\Zz(C_{-+}(\alpha;\beta)):=\phi_\beta^\alpha$$
and since $C_{--}(\alpha;\beta)$ and $C_{++}(\alpha;\beta)$ can be optained by gluing $C_{-+}(\alpha^{-1};\beta)$ to $C_{--}(\alpha;1_x)$ and $C_{++}(\alpha;1_x)$ define
\begin{eqnarray*}\Zz(C_{--}(\alpha;\beta))&:=&\eta_{\alpha} (L_\alpha \ot \phi^{\beta\alpha^{-1}\beta^{-1}}_{\beta^{-1}}) =\eta_{\beta\alpha\beta^{-1}} (\phi_{\beta}^\alpha \ot L_{\beta\alpha^{-1}\beta^{-1}})\\
\Zz(C_{++}(\alpha;\beta))&:=&(\phi^{\alpha^{-1}}_\beta \ot L_\alpha) \coev_\alpha= (L_{\beta\alpha^{-1}\beta^{-1}}\ot \phi^{\beta\alpha\beta^{-1}}_{\beta^{-1}}) \coev_{\beta\alpha\beta^{-1}}.
\end{eqnarray*}

\[
\begin{tikzpicture}
\draw[fill=gray!30,even odd rule]  (0,0) circle (3cm) 
                                (0,0) circle (1cm);
\draw (0,0) circle (2cm);
\draw[decorate,decoration={markings,mark=at position 6.3 cm with
{\arrow[line width=1mm]{stealth}};}] (0,0) circle (2cm) node at (-2.5,0) {$\alpha^{-1}$};
\draw[decorate,decoration={markings,mark=at position 9.3 cm with
{\arrow[line width=1mm]{stealth}};}] (0,0) circle (3cm) node at (-4,0) {$\beta\alpha^{-1}\beta^{-1}$};
\draw[decorate,decoration={ markings,mark=at position 3cm with
{\arrowreversed[line width=1mm]{stealth}};}] (0,0) circle (1cm) node at (-1.5,0) {$\alpha$};
\draw (0,-2) node{$\bullet$};
\draw (0,-1) node{$\bullet$};
\draw[decorate,decoration={ markings,mark=at position 0.6cm with
{\arrow[line width=1mm]{stealth}};}] (0,-2) --(0,-1)  ;
\draw[decorate,decoration={ markings,mark=at position 0.6cm with
{\arrow[line width=1mm]{stealth}};}] (0,-3) --(0,-2)  ;
\draw (0,-2) --(0,-1) node at (-0.5,-1.5) {$1_x$} ;
\draw (0,-3) --(0,-2) node at (-0.5,-2.5) {$\beta$} ;
\draw (0,-3.5) node {$C_{--}(\alpha;\beta)$};

\draw[fill=gray!30,even odd rule]  (8,0) circle (3cm) 
                                (8,0) circle (1cm);
\draw (8,0) circle (2cm);
\draw[decorate,decoration={markings,mark=at position 9.5 cm with
{\arrowreversed[line width=1mm]{stealth}};}] (8,0) circle (3cm) node at (4,0) {$\beta\alpha^{-1}\beta^{-1}$};
\draw[decorate,decoration={markings,mark=at position 6.5 cm with
{\arrowreversed[line width=1mm]{stealth}};}] (8,0) circle (2cm) node at (5.5,0) {$\alpha^{-1}$};
\draw[decorate,decoration={ markings,mark=at position 3cm with
{\arrow[line width=1mm]{stealth}};}] (8,0) circle (1cm) node at (6.5,0) {$\alpha$};
\draw (8,-2) node{$\bullet$};
\draw (8,-1) node{$\bullet$};
\draw[decorate,decoration={ markings,mark=at position 0.6cm with
{\arrow[line width=1mm]{stealth}};}] (8,-2) --(8,-1)  ;
\draw[decorate,decoration={ markings,mark=at position 0.6cm with
{\arrow[line width=1mm]{stealth}};}] (8,-3) --(8,-2)  ;
\draw (8,-2) --(8,-1) node at (7.5,-1.5) {$1_x$} ;
\draw (8,-3) --(8,-2) node at (7.5,-2.5) {$\beta$} ;
\draw (8,-3.5)  node {$C_{++}(\alpha;\beta)$};
\end{tikzpicture}
\]

The equality of the two expressions follows from the axioms of a crossed loop Frobenius $(\Gg,\Vv)$-category
\begin{eqnarray*}
\eta_{\beta\alpha\beta^{-1}} (\phi_{\beta}^\alpha \ot L_{\beta\alpha^{-1}\beta^{-1}})&=&\nu_y m_{\beta\alpha\beta^{-1},\beta\alpha^{-1}\beta^{-1}} (\phi_{\beta}^\alpha \ot L_{\beta\alpha^{-1}\beta^{-1}})\\
&=&\nu_y \phi^{1_x}_\beta m_{\alpha,\alpha^{-1}} (\phi_{\beta^{-1}}^{\beta\alpha\beta^{-1}} \ot \phi_{\beta^{-1}}^{\beta\alpha^{-1}\beta^{-1}} )(\phi_{\beta}^\alpha \ot L_{\beta\alpha^{-1}\beta^{-1}})\\
&=&\nu_x m_{\alpha,\alpha^{-1}} (\phi_{\beta^{-1}\beta}^{\alpha}\ot \phi_{\beta^{-1}}^{\beta\alpha^{-1}\beta^{-1}} )\\
&=&\eta_{\alpha} (L_\alpha \ot \phi^{\beta\alpha^{-1}\beta^{-1}}_{\beta^{-1}}) 
\end{eqnarray*}
and dually using comultiplication and coevaluation. \\
Every $(X,Y)$-disc with two holes $D_{\epsilon,\mu,\nu}$ can be brought to either one of the 4 forms: $D_{---}(\alpha,\beta;$ $\rho,\delta)$, 
$D_{--+}(\alpha,\beta;\rho,\delta)$, $D_{++-}(\alpha,\beta;\rho,\delta)$ or $D_{+++}(\alpha,\beta;\rho,\delta)$ via
the $(X,Y)$-homeomorphisms
$$D_{\epsilon,\mu,\nu} (\alpha,\beta;\rho,\delta)\simeq D_{\mu,\nu,\epsilon} (\beta,\gamma;\rho^{-1}\delta, \rho^{-1})\simeq D_{\nu,\epsilon,\mu,} (\gamma,\alpha; \delta^{-1},\delta^{-1}\rho)$$
with $\gamma=(\rho\alpha^{-\epsilon}\rho^{-1}\delta\beta^{-\mu}\delta^{-1})^{\nu}$.
We have 
\begin{eqnarray*}
\Zz(D_{---}(\alpha,\beta;\rho,\delta))&=&\eta_{\gamma^{-1}}(m_{\rho\alpha\rho^{-1},\delta\beta\delta^{-1}}\ot L_\gamma)(\phi^\alpha_\rho \ot \phi^\beta_\delta \ot L_\gamma),\\
\Zz(D_{--+}(\alpha,\beta;\rho,\delta))&=&m_{\rho\alpha\rho^{-1},\delta\beta\delta^{-1}}(\phi^\alpha_\rho \ot \phi^\beta_\delta),\\
\Zz(D_{++-}(\alpha,\beta;\rho,\delta))&=&(\phi^{\delta\beta\delta^{-1}}_{\delta^{-1}}\ot \phi^{\rho\alpha\rho^{-1}}_{\rho^{-1}}) \Delta_{\delta\beta\delta^{-1},\rho\alpha\rho^{-1}},\\
\Zz(D_{+++}(\alpha,\beta;\rho,\delta))&=&(\phi^{\delta\beta\delta^{-1}}_{\delta^{-1}}\ot \phi^{\rho\alpha\rho^{-1}}_{\rho^{-1}}\ot L_\gamma) (\Delta_{\delta\beta\delta^{-1},\rho\alpha\rho^{-1}} \ot L_\gamma) \coev_{\gamma}.\\ 
\end{eqnarray*}
This can be seen by gluing together $D_{--+}(\alpha,\beta;1_x,1_x)$ and $D_{++-}(\alpha,\beta;1_x,1_x)$ with cylinders $C_{\epsilon,\mu}$ with corresponding signs. Since every $(X,Y)$-surface can be disected into $(X,Y)$-discs with $\leq 2$ holes, this shows that any morphism in the image of $\Zz$ can be obtained by composition of structural morphisms and the functor $F$ is fully faithful. \\
We show that this assignement is topologically invariant and since the mapping class group is generated by the Dehn twists along boundary parallel circles and the reflection, it suffices to show invariance under these two operations. 
The Dehn twist is an $(X,Y)$-homeomorphism $C_{\epsilon,\mu}(\alpha;\beta) \to C_{\epsilon,\mu}(\alpha;\beta\alpha)$ hence both cylinders must be mapped to the same map:
\begin{equation*}
\Zz(C_{-+}(\alpha;\alpha\beta))=\phi^\alpha_{\beta\alpha}=\phi^\alpha_\beta \phi^\alpha_\alpha=\phi^\alpha_\beta=\Zz(C_{-+}(\alpha;\beta))
\end{equation*}
and since the rest of the cylinders are defined using $C_{-+}(\alpha;\beta)$ this holds for all $\epsilon,\mu=\pm$.
The reflection is an $(X,Y)$-homeomorphism $C_{\epsilon,\epsilon}(\alpha;\beta)\to C_{\epsilon,\epsilon}(\beta\alpha^{-1}\beta^{-1};\beta^{-1})$ given by a permutation of the boundary components and the product of the orientation reversing involutions $\mathbb{S}\to \mathbb{S}, s\mapsto -\bar s$ and $[0,1]\to [0,1], t\mapsto 1-t$.  
We get
\begin{eqnarray*}
\Zz(C_{--}(\beta\alpha^{-1}\beta^{-1};\beta^{-1}))\sigma&=&\eta_{\beta\alpha^{-1}\beta^{-1}} (L_{\beta\alpha^{-1}\beta^{-1}}\ot \phi^{\alpha}_{\beta})\sigma\\
&=&\eta_{\beta\alpha^{-1}\beta^{-1}} \sigma (\phi^{\alpha}_{\beta}\ot L_{\beta\alpha^{-1}\beta^{-1}})\\
&=&\eta_{\beta\alpha\beta^{-1}}(\phi^{\alpha}_{\beta}\ot L_{\beta\alpha^{-1}\beta^{-1}})= \Zz(C_{--}(\alpha;\beta)),
\end{eqnarray*}
which has been shown above and equivalently for $C_{++}(\beta\alpha^{-1}\beta^{-1},\beta^{-1}).$
The invariance of the $(X,Y)$-disc with two holes under the Dehn twists follows form  the fact that we defined them by gluing $C_{-+}(\alpha;\rho)$ and $C_{-+}(\beta;\delta)$ at the inner circles, which satisfy the invariance. For the Dehn twist along the outer circle observe that multiplication and comultiplication preserve the action of $\phi$ and $\eta, \coev$ are invariant under it. Hence, we can reduce this to invariance of cylinders. 
The reflection of $D_{---}(\alpha,\beta;\rho,\delta)$ is $D_{---}(\beta,\gamma; \rho^{-1}\delta, \rho^{-1}\rho)$ and using $\rho^{-1}\delta\beta\delta^{-1}\rho\rho^{-1}\gamma\rho=\rho^{-1}\delta\beta\delta^{-1}(\delta \beta^{-1} \delta^{-1}\rho\alpha^{-1}\rho^{-1})\rho=\alpha^{-1}$ we deduce:

\begin{eqnarray*}
& & D_{---} (\beta,\gamma;\rho^{-1}\delta, \rho^{-1}) \sigma_{(132)}\\
&=&\eta_{\alpha^{-1}}(m_{\rho^{-1}\delta\beta\delta^{-1}\rho,\rho^{-1}\gamma\rho} \ot L_\alpha)(\phi^\beta_{\rho^{-1}\delta} \ot \phi^{\gamma}_{\rho^{-1}} \ot L_\alpha)\sigma_{(132)}\\
&=&\eta_{\alpha^{-1}}(m_{\alpha^{-1}\rho^{-1}\gamma^{-1}\rho,\rho^{-1}\gamma\rho} \ot L_\alpha)(\phi^\beta_{\rho^{-1}\delta} \ot \phi^{\gamma}_{\rho^{-1}} \ot L_\alpha)\sigma_{(132)}\\
&=&\eta_{\alpha^{-1}\rho^{-1}\gamma^{-1}\rho}(L_{\alpha^{-1}\rho^{-1}\gamma^{-1}\rho} \ot m_{\rho^{-1}\gamma\rho,\alpha})\sigma_{(132)}(L_\alpha \ot \phi^\beta_{\rho^{-1}\delta} \ot \phi^{\gamma}_{\rho^{-1}})\\
&=&\eta_{\alpha^{-1}\rho^{-1}\gamma^{-1}\rho}(L_{\alpha^{-1}\rho^{-1}\gamma^{-1}\rho} \ot m_{\rho^{-1}\gamma\rho\alpha\rho^{-1}\gamma^{-1}\rho,\rho^{-1}\gamma\rho})(L_{\alpha^{-1}\rho\gamma^{-1}\rho} \ot \phi^\alpha_{\rho^{-1}\gamma\rho} \ot L_{\rho^{-1}\gamma\rho})\\
& &(\sigma \ot L_{\rho^{-1}\gamma\rho})(L_\alpha \ot \phi^\beta_{\rho^{-1}\delta} \ot \phi^{\gamma}_{\rho^{-1}})\\
&=&\eta_{\rho^{-1}\gamma^{-1}\rho}(m_{\alpha^{-1}\rho^{-1}\gamma^{-1}\rho,\rho^{-1}\gamma\rho\alpha\rho^{-1}\gamma^{-1}\rho}\ot L_{\rho^{-1}\gamma\rho})(\sigma \ot L_{\rho^{-1}\gamma\rho})\\
& &( \phi^\alpha_{\rho^{-1}\gamma\rho}\ot \phi^\beta_{\rho^{-1}\delta} \ot \phi^{\gamma}_{\rho^{-1}})\\
&=&\eta_{\rho^{-1}\gamma^{-1}\rho}(m_{\alpha,\alpha^{-1}\rho^{-1}\gamma^{-1}\rho}\ot L_{\rho^{-1}\gamma\rho})(\phi^{\rho^{-1}\gamma\rho\alpha\rho^{-1}\gamma^{-1}\rho}_{\alpha^{-1}\rho^{-1}\gamma^{-1}\rho}\phi^\alpha_{\rho^{-1}\gamma\rho} \ot \phi^\beta_{\rho^{-1}\delta}\ot \phi^{\gamma}_{\rho^{-1}})\\
&=&\eta_{\gamma^{-1}}(\phi^{\rho^{-1}\gamma\rho}_{\rho} m_{\alpha,\alpha^{-1}\rho^{-1}\gamma^{-1}\rho}\ot L_{\gamma})( \phi^\alpha_{\alpha^{-1}} \ot \phi^\beta_{\rho^{-1}\delta} \ot L_\gamma)\\
&=&\eta_{\gamma^{-1}}(m_{\rho\alpha\rho^{-1},\rho\alpha^{-1}\rho^{-1}\gamma^{-1}}\ot L_{\gamma})( \phi^\alpha_\rho \ot \phi^{\alpha^{-1}\rho^{-1}\gamma^{-1}\rho}_{\rho}\phi^\beta_{\rho^{-1}\delta} \ot L_\gamma)\\
&=&\eta_{\gamma^{-1}}(m_{\rho\alpha\rho^{-1},\rho\alpha^{-1}\rho^{-1}\gamma^{-1}}\ot L_{\gamma})( \phi^\alpha_\rho \ot \phi^{\rho^{-1}\delta\beta\delta^{-1}\rho}_{\rho}\phi^\beta_{\rho^{-1}\delta} \ot L_\gamma)\\
&=&\eta_{\gamma^{-1}}(L_{\rho\alpha\rho^{-1}\delta\beta\delta^{-1}}\ot L_\gamma) (m_{\rho\alpha\rho^{-1},\delta\beta\delta^{-1}}\ot L_\gamma)(\phi^\alpha_\rho \ot \phi^\beta_\delta \ot L_\gamma)\\
&=& D_{---}(\alpha,\beta; \rho,\delta),
\end{eqnarray*}
where we denote by $\sigma_{(132)}=\sigma_{A,B\ot C}:A\ot B\ot C\to B\ot C\ot A$ for all object $A,B,C \in \Vv$.
Similarly, the invariance of  $D_{+++}(\alpha,\beta;\rho,\delta)$ under the reflexion is shown. \\
We finish by showing this assignement is invariant under gluing. The gluing of $C_{\epsilon,\mu}(\alpha;\beta)$ to $C_{-\mu,\nu}(\beta\alpha^{-\epsilon\mu}\beta^{-1};\delta)$ yields $C_{\epsilon,\nu}(\alpha,\delta\beta)$. Without loss of generality assume that $\mu=+$, since both cases are topologically equivalent.
\begin{itemize}
\item $\epsilon=-,\nu=+$: 
$$\Zz\left(C_{-,+}(\beta\alpha\beta^{-1};\delta)\coprod C_{-+}(\alpha;\beta)\right)=\phi^{\beta\alpha\beta^{-1}}_\delta \phi_\beta^\alpha=\phi^\alpha_{\delta\beta}=\Zz(C_{-,+}(\alpha;\delta\beta)).$$ 
\item $\epsilon=-,\nu=-$: 
\begin{eqnarray*}
\Zz\left(C_{-,-}(\beta\alpha\beta^{-1};\delta)\coprod C_{-+}(\alpha;\beta)\right)
&=&\eta_{\delta\beta\alpha\beta^{-1}\delta^{-1}}(\phi^{\beta\alpha\beta^{-1}}_\delta \ot L_{\delta\beta\alpha^{-1}\beta^{-1}\delta^{-1}})(\phi_\beta^\alpha \ot L_{\delta\beta\alpha^{-1}\beta^{-1}\delta^{-1}})\\
&=&\eta_{\delta\beta\alpha\beta^{-1}\delta^{-1}}(\phi_{\delta\beta}^\alpha \ot L_{\delta\beta\alpha^{-1}\beta^{-1}\delta^{-1}})\\
&=&\Zz(C_{-,-}(\alpha;\delta\beta)).
\end{eqnarray*}
\item $\epsilon=+,\nu=+$: 
\begin{eqnarray*}
\Zz\left(C_{-,+}(\beta\alpha^{-1}\beta^{-1};\delta)\coprod C_{++}(\alpha;\beta)\right)
&=&(\phi^{\beta\alpha^{-1}\beta^{-1}}_\delta  \ot L_\alpha) (\phi^{\alpha^{-1}}_\beta \ot L_\alpha) \coev_{\alpha}\\
&=&(\phi^{\alpha^{-1}}_{\delta\beta} \ot L_\alpha) \coev_{\alpha}\\
&=&\Zz(C_{+,+}(\alpha,\delta\beta)).
\end{eqnarray*}
\item $\epsilon=+,\nu=-$: 
\begin{eqnarray*}
& &\Zz\left(C_{-,-}(\beta\alpha^{-1}\beta^{-1};\delta)\coprod C_{++}(\alpha;\beta)\right)\\
&=&(\eta_{\beta\alpha^{-1}\beta^{-1}} \ot L_\alpha)(L_{\beta\alpha^{-1}\beta^{-1}} \ot \phi^{\delta\beta\alpha\beta^{-1}\delta^{-1}}_{\delta^{-1}} \ot L_\alpha) (\sigma \ot L_\alpha)\\
& &(L_{\delta\beta\alpha\beta^{-1}\delta^{-1}}\ot L_{\beta\alpha^{-1}\beta} \ot \phi^{\beta\alpha\beta^{-1}}_{\beta^{-1}}) ( L_{\delta\beta\alpha\beta^{-1}\delta^{-1}}\ot\coev_{\beta\alpha\beta^{-1}}) \\
&=&\phi^{\beta\alpha\beta^{-1}}_{\beta^{-1}}(\eta_{\beta\alpha\beta^{-1}} \ot L_{\beta\alpha\beta^{-1}})( L_{\beta\alpha\beta^{-1}}\ot\coev_{\beta\alpha\beta^{-1}}) \phi^{\delta\beta\alpha\beta^{-1}\delta^{-1}}_{\delta^{-1}}\\
&=&\phi^{\beta\alpha\beta^{-1}}_{\beta^{-1}}\phi^{\delta\beta\alpha\beta^{-1}\delta^{-1}}_{\delta^{-1}}=\phi_{(\delta\beta)^{-1}}^{\delta\beta\alpha\beta^{-1}\delta^{-1}}\\
&=&\Zz(C_{-,+}(\delta\beta\alpha\beta^{-1}\delta^{-1};(\delta\beta)^{-1}))\\
&=&\Zz(C_{+,-}(\alpha;\delta\beta)).
\end{eqnarray*}
\end{itemize}
To complete the proof we need gluings $D_{\epsilon,\mu,\nu}$ to $C_{\pi,\omega}$. If the cylinder is glued to any of the inner circles, the relation for composition of the crossings yields the equality. So let us only consider gluings along the outer circle. Further, if we consider $C_{-+}$ or $C_{+-}$ we can use compatibility of $m,\Delta$ with the crossing or invariance of $\eta,\coev$ under the crossing to prove the equality of the two sides. Hence we concentrate our efforts on the case, where $\pi=\omega$ 
\begin{itemize}
\item $\epsilon=\mu=\nu=-,\pi=\omega=+$:
\begin{eqnarray*}
& & \Zz\left(D_{-,-,-}(\alpha,\beta;\rho,\delta)\coprod C_{++}(\gamma;\xi)\right)=\\
&=&(\eta_{\gamma^{-1}}\ot L_{\xi\gamma^{-1}\xi^{-1}})(L_{\gamma^{-1}} \ot \sigma)(L_{\gamma^{-1}}\ot \phi^{\gamma^{-1}}_\xi \ot L_\gamma) (L_{\gamma^{-1}} \ot \coev_{\gamma})m_{\rho\alpha\rho^{-1},\delta\beta\delta^{-1}} (\phi^\alpha_\rho \ot \phi^\beta_\delta)\\
&=&\phi^{\gamma^{-1}}_\xi (\eta_{\gamma^{-1}}\ot L_\gamma)(L_{\gamma^{-1}} \ot \coev_{\gamma^{-1}})m_{\rho\alpha\rho^{-1},\delta\beta\delta^{-1}}(\phi^\alpha_\rho \ot \phi^\beta_\delta)\\
&=&\phi^{\gamma^{-1}}_\xi m_{\rho\alpha\rho^{-1},\delta\beta\delta^{-1}}(\phi^\alpha_\rho \ot \phi^\beta_\delta)\\
&=&m_{\xi\rho\alpha\rho^{-1}\xi^{-1},\xi\delta\beta\delta^{-1}\xi^{-1}}(\phi^\alpha_{\xi\rho} \ot \phi^\beta_{\xi\delta})\\
&=&\Zz(D_{-,-,+}(\alpha,\beta;\xi\rho,\xi\delta)).
\end{eqnarray*}
\item $\epsilon=\mu=-, \nu=+,\pi=\omega=-$:
\begin{eqnarray*}
& & \Zz\left(D_{-,-,+}(\alpha,\beta;\rho,\delta)\coprod C_{--}(\gamma;\xi)\right)=\\
&=&\eta_{\gamma} (L_\gamma \ot \phi^{\xi\gamma^{-1}\xi^{-1}}_{\xi^{-1}}) (m_{\rho\alpha\rho^{-1},\delta\beta\delta^{-1}} \ot L_{\xi\gamma^{-1}\xi^{-1}})(\phi^\alpha_\rho \ot \phi^\beta_\delta\ot L_{\xi\gamma^{-1}\xi^{-1}})\\
&=&\eta_{\xi\gamma\xi^{-1}} (\phi^\gamma_\xi \ot L_{\xi\gamma^{-1}\xi^{-1}}) (m_{\rho\alpha\rho^{-1},\delta\beta\delta^{-1}} \ot L_{\xi\gamma^{-1}\xi^{-1}})(\phi^\alpha_\rho \ot \phi^\beta_\delta\ot L_{\xi\gamma^{-1}\xi^{-1}})\\
&=&\eta_{\xi\gamma\xi^{-1}}  (m_{\xi\rho\alpha\rho^{-1}\xi^{-1},\xi\delta\beta\delta^{-1}\xi^{-1}} \ot L_{\xi\gamma^{-1}\xi^{-1}})(\phi^\alpha_{\xi\rho} \ot \phi^\beta_{\xi\delta} \ot L_{\xi\gamma^{-1}\xi^{-1}})\\
&=&\Zz(D_{-,-,-}(\alpha,\beta;\xi\rho,\xi\delta)).
\end{eqnarray*}
\item $\epsilon=\mu=+=\nu=+,\pi=\omega=-$:
\begin{eqnarray*}
& & \Zz\left(D_{+,+,+}(\alpha,\beta;\rho,\delta)\coprod C_{--}(\gamma;\xi)\right)=\\
&=&(L_\alpha\ot L_\beta \ot \eta_{\gamma}) (\phi^{\delta\beta\delta^{-1}}_{\delta^{-1}}\ot \phi^{\rho\alpha\rho^{-1}}_{\rho^{-1}}\ot L_\gamma\ot \phi^{\xi\gamma^{-1}\xi^{-1}}_{\xi^{-1}})\\
&&(\Delta_{\delta\beta\delta^{-1},\rho\alpha\rho^{-1}} \ot L_\gamma \ot L_{\xi\gamma^{-1}\xi^{-1}}) (\coev_{\gamma}\ot L_{\xi\gamma^{-1}\xi^{-1}})\\
&=&(\phi^{\delta\beta\delta^{-1}}_{\delta^{-1}}\ot \phi^{\rho\alpha\rho^{-1}}_{\rho^{-1}})\Delta_{\delta\beta\delta^{-1},\rho\alpha\rho^{-1}} (L_{\gamma^{-1}} \ot \eta_\gamma) (\coev_{\gamma}\ot L_{\gamma^{-1}}) \phi^{\xi\gamma^{-1}\xi^{-1}}_{\xi^{-1}}\\
&=&(\phi^{\delta\beta\delta^{-1}}_{\delta^{-1}}\ot \phi^{\rho\alpha\rho^{-1}}_{\rho^{-1}})\Delta_{\delta\beta\delta^{-1},\rho\alpha\rho^{-1}} \phi^{\xi\gamma^{-1}\xi^{-1}}_{\xi^{-1}}\\
&=&(\phi^{\xi\delta\beta\delta^{-1}\xi^{-1}}_{(\xi\delta)^{-1}}\ot \phi^{\xi\rho\alpha\rho^{-1}\xi^{-1}}_{(\xi\rho)^{-1}})\Delta_{\xi\delta\beta\delta^{-1}\xi^{-1},\xi\rho\alpha\rho^{-1}\xi^{-1}}\\
&=&\Zz(D_{+,+,-}(\alpha,\beta;\xi\rho,\xi\delta)).
\end{eqnarray*}
\item $\epsilon=\mu=+=\nu=-,\pi=\omega=+$:
\begin{eqnarray*}
& & \Zz\left(D_{+,+,-}(\alpha,\beta;\rho,\delta)\coprod C_{++}(\gamma;\xi)\right)=\\
&=&(L_{\xi\gamma^{-1}\xi^{-1}} \ot \phi^{\delta\beta\delta^{-1}}_{\delta^{-1}}\ot \phi^{\rho\alpha\rho^{-1}}_{\rho^{-1}})(L_{\xi\gamma^{-1}\xi}\ot \Delta_{\delta\beta\delta^{-1},\rho\alpha\rho^{-1}})(\phi^{\gamma^{-1}}_\xi  \ot L_\gamma)\coev_\gamma\\
&=&(L_{\xi\gamma^{-1}\xi^{-1}} \ot \phi^{\delta\beta\delta^{-1}}_{\delta^{-1}}\ot \phi^{\rho\alpha\rho^{-1}}_{\rho^{-1}})(L_{\xi\gamma^{-1}\xi}\ot \Delta_{\delta\beta\delta^{-1},\rho\alpha\rho^{-1}})(L_{\gamma^{-1}}\ot \phi^{\gamma}_\xi )\sigma \coev_{\xi\gamma^{-1}\xi^{-1}}\\
&=&(\phi^{\xi\delta\beta\delta^{-1}\xi^{-1}}_{(\xi\delta)^{-1}}\ot \phi^{\xi\rho\alpha\rho^{-1}\xi^{-1}}_{(\xi\rho)^{-1}}\ot L_{\xi\gamma\xi^{-1}}) (\Delta_{\xi\delta\beta\delta^{-1}\xi^{-1},\xi\rho\alpha\rho^{-1}\xi^{-1}} \ot L_{\xi\gamma\xi^{-1}}) \coev_{\xi\gamma\xi^{-1}}\\
&=&\Zz(D_{+,+,+}(\alpha,\beta;\xi\rho,\xi\delta)).
\end{eqnarray*}
\end{itemize}
For arbitrary $(X,Y)$-surfaces we conclude the proof analogously to the proof of Turaev \cite[Thm 3.1.]{Tu1} stating the independence of the splitting system of loops.
\end{proof}

\section{Outlook}
The structure of a crossed loop Frobenius $(\Gg,\Vv)$-category appears somehow incomplete, since it uses the full groupoid for the crossings, however only loops on the indices of the objects. An attempt to solve this asymmetry is to define a category of open/closed 2-dimensional $(X,Y)$-cobordisms, whose objects are given by open or closed strings together with appropriate homotopy classes of maps to $Y$. In order to define morphisms in this category, one needs to consider 2-manifolds with corners with additional data. Given such a symmetric monoidal category $\Cob_{2,oc}^{X,Y}$, one might define \textit{open/closed HQFTs} as symmetric monoidal functors $\Zz:\Cob_{2,oc}^{X,Y}\to \Vv$. For $X=Y=\{*\}$ this theory should reduce to open/closed TQFTs, as defined by Lauda and Pfeiffer in \cite{LP}, who classified them by knowledgeable Frobenius algebras. For general target pairs of spaces, one therefore can expect the resulting algebraic structure to be some sort of \textit{crossed knowledgable Frobenius $(\Gg,\Vv)$-category}. The study of such a theory of open/closed HQFTs will be subject to an upcoming paper.


\newpage

\begin{thebibliography}{99}
\bibitem{Ab}
L.\ Abrams, Two-dimensional topological quantum field theories and Frobenius algebras, {\em Journal of Knot Theory and Its Ramifications} {\bf 5} (5), (1996), 569--587.

\bibitem{BT}
M.\ Brightwell and P.\ Turner, Representations of the homotopy surface category of a simply connected space, {\em Journal of Knot theory and its Ramifications} {\bf 9} (7), (2000), 855--864. 

\bibitem{Kassel}
C.\ Kassel, Quantum Groups, {\em Springer-Verlag}, New York 1995.

\bibitem{Kock} 
J.\ Kock, Frobenius Algebras and 2D Topological Quantum Field Theories, {\em Cambridge University Press}, Cambridge (2003). 

\bibitem{LP}
A.\ Lauda and H.\ Pfeiffer, Open-closed strings: Two-dimensional extended TQFTs and Frobenius algebras, {\em Topology and its applications} {\bf 155} (7), (2008), 623--666.

\bibitem{MacLane}
S.\ Mac Lane, Categories for the Working Mathematician, Graduate Texts in Mathematics 5, {\em  Springer-Verlag}, New York 1998.


\bibitem{PT}
T.\ Porter and V.\ Turaev, Formal Homotopy Quantum Field Theories, I: Formal Maps and Crossed $\Cc$-algebras, {\em Journal of Homotopy and Related Strucutures} {\bf 3} (1), (2006).

\bibitem{R}
G.\ Rodrigues, Homotopy Quantum Field Theories and the Homotopy Cobordism Category in Dimension 1 + 1, {\em J. Knot Theory and its Ramifications} {\bf 12}, (2003), 287--317. 

\bibitem{ST}
M.D.\ Staic and V.\ Turaev, Remarks on 2-dimensional HQFT's, {\em Algebraic \& Geometric Topology} {\bf 10}, (2010), 1367--1393.

\bibitem{Tu}
V.\ Turaev, Homotopy field theory in dimension 2 and group-algebras, {\em http://arxiv.org/abs/math/9910010}.

\bibitem{Tu1}
V.\ Turaev, Homotopy Quantum Field Theory, EMS Tracts in Mathematics, 2010.

\end{thebibliography}
\end{document}